\documentclass[11pt]{amsart}
\usepackage[utf8]{inputenc}
\usepackage[text={6.1in,8.5in},centering]{geometry}
\usepackage{amssymb,amsmath,amsthm,bm,fdsymbol}
\usepackage{tikz}
\usepackage[square,numbers]{natbib}

\usetikzlibrary{patterns}
\usetikzlibrary{fadings}

\newtheorem{thm}[equation]{Theorem}
\newtheorem*{thm*}{Theorem}
\newtheorem*{lem*}{Lemma}

\newtheorem{lem}[equation]{Lemma}
\newtheorem{prop}[equation]{Proposition}
\newtheorem*{prop*}{Proposition}
\newtheorem{cor}[equation]{Corollary}
\newtheorem{conj}[equation]{Conjecture}

\theoremstyle{definition}
\newtheorem{defn}[equation]{Definition}
\newtheorem{rmk}[equation]{Remark}
\newtheorem{ex}[equation]{Example}

\numberwithin{equation}{section}

\newcommand{\I}{\mathcal{I}}

\newcommand{\noep}{\delta}
\newcommand{\edual}{\tilde{e}}
\newcommand{\espir}{e}
\newcommand{\vspir}{v}

\newcommand{\ek}{\check{e}}
\newcommand{\vk}{\check{v}}
\newcommand{\nk}{\check{n}}

\title{Extremal $\{p, q\}$-Animals}
\author[G.~Malen]{Greg Malen}
\address[G.~Malen]{Department of Mathematics, Union College, Schenectady, NY, United States}
\email{maleng@union.edu}

\author[\'E.~Rold\'an]{\'Erika Rold\'an}
\address[\'E.~Rold\'an]{Zentrum Mathematik, TU M\"unchen, Garching b. M\"unchen, Germany}
\email{erika.roldan@ma.tum.de}

\author[R.~Toal\'a-Enr\'iquez]{Rosemberg Toal\'a-Enr\'iquez}
\address[R.~Toal\'a-Enr\'iquez]{Adjunct Researcher, Topological Data Analysis Group, Center for research in Mathematics, Guanajuato, M\'exico.}
\email{toala@cimat.mx}

\begin{document}
\begin{abstract}
An animal is a planar shape formed by attaching congruent regular polygons, known as tiles, along their edges. In this paper, we study extremal animals defined on regular tessellations of the plane. In 1976, Harary and Harborth studied animals in the Euclidean cases, finding extremal values for their vertices, edges, and tiles, when any one of these parameters is fixed. Here, we generalize their results to hyperbolic animals. For each hyperbolic tessellation, we exhibit a sequence of spiral animals and prove that they attain the minimal numbers of edges and vertices within the class of animals with $n$ tiles. In their conclusions, Harary and Harborth also proposed the question of enumerating extremal animals with a fixed number of tiles. This question has previously only been considered for Euclidean animals. As a first step in solving this problem, we find special sequences of extremal animals that are \textit{unique} extremal animals, in the sense that any animal with the same number of tiles which is distinct up to isometries can't be extremal.
\end{abstract}
\keywords{Extremal combinatorics, enumerative combinatorics, polyforms, polyominoes, polyiamonds, hexiamonds, animals, hyperbolic tessellations, isoperimetric inequalities. \\
* See \cite{beautifulanimals} for our interactive applets to explore extremal animals and their graph parameters.}

\maketitle

\section{Introduction}

An \emph{animal} is a planar shape with connected interior formed by gluing together a finite number of congruent regular polygons along their edges. Here, we study animals that have extremal combinatorial properties such as having maximally many shared (or interior) edges for a given number of polygons. We refer to animals attaining these optimal values as \emph{extremal animals}. In the Euclidean plane, extremal animals have been well studied and it is known that they exhibit interesting geometric and combinatorial properties such as being isoperimetric \cite{harary1976extremal}, giving contact numbers for optimal disk packing configurations \cite{harborth1974losung}, and providing the right shape to produce animals with maximally many holes \cite{KaR,MaR, malen2021extremalI, malen2021extremalII}. 

In the hyperbolic plane, combinatorial and geometric properties of extremal animals have not been examined as closely. However, the relevance of extremal animals in this setting to other extremal combinatorial problems has already been established. These include optimal disk packing problems \cite{bowen2002circle}, the calculation of Cheeger constants \cite{higuchi2003isoperimetric}, the establishment of the exponential growth constant of regular hyperbolic tessellations \cite{keller2008geometric}, and the implementation of algorithms to sample certain hyperbolic animals to approximate hyperbolic percolation thresholds \cite{mertens2017percolation}.

Recently, extremal hyperbolic animals have been used to model and study structures in applied mathematics, such as in crystallography theory for materials with hyperbolic symmetries \cite{boettcher2021crystallography}, and in the architecture of circuit quantum electrodynamics \cite{kollar2019hyperbolic}. More generally,
models based on hyperbolic tessellations have also been used in areas such as Bloch band theory in physics \cite{maciejko2020hyperbolic}, and in quantum error correcting and quantum storage codes \cite{breuckmann2017hyperbolic, jahn2021holographic}, among others. Therefore, having a better understanding of extremal hyperbolic animals might lead to interesting insights, results, techniques, and new research questions in these areas. 

The families of animals that we study in this paper are formed by finite subsets of tiles of regular Euclidean and hyperbolic tessellations. These tessellations can be parameterized using the Schläfli symbol $\{p,q\}$, where $p$ denotes the number of sides of the regular polygon forming the tessellation and $q$ is the number of edges or tiles meeting at each vertex. It is well known that if $(p-2)(q-2)> 4$, $=4$, or $<4$, then the tessellation corresponds to the geometry of the hyperbolic plane, the Euclidean plane, or the sphere, respectively. We call an animal living in a $\{p,q\}$-tessellation a $\{p,q\}$-animal. The most well-known families of animals are those living in the Euclidean tessellations, namely for $\{p,q\}$ either $\{3,6\}$, $\{4,4\}$, or $\{6,3\}$. Formed by triangles, squares, or hexagons, respectively, these animals are also more commonly known as polyiamonds, polyominoes, and polyhexes in the literature.


In 1976 \cite{harary1976extremal}, Harary and Harborth studied extremal animals in the three Euclidean cases. 
More specifically, they gave algebraic expressions and bounds for the number of interior, exterior, and total numbers of both vertices and edges that an animal with a fixed number of tiles can have. Then they constructed a family of \textit{spiral} animals which attain optimal values. Here, we extend Harary and Harborth's analysis to extremal hyperbolic animals. First, we define and analyze analogous sequences of \textit{spiral} $\{p,q\}$-animals for $(p-2)(q-2) \geq 4$ and provide algebraic expressions for the number of vertices and edges in terms of a particular recursive sequence of integers. Then we prove that these \textit{spiral} $\{p,q\}$-animals are indeed extremal. As a corollary of our results, we recover the formulas derived in \cite{harary1976extremal} for the Euclidean cases.


We also explore enumeration of extremal hyperbolic animals. Exact enumeration of $\{p,q\}$-animals is only known for a few $\{p,q\}$ pairs, and even then only for animals with a very small number of tiles. For instance, the number of fixed $\{4,4\}$-animals with $n$ tiles is only known for $n \leq 56$ \cite{jensen2003counting}; and otherwise very little is known (see, for example, the entries of the OEIS A001207, A001420, A000228, A000577, A119611 \cite{sloane2003line} for values in the $\{3,6\}$-, $\{6,3\}$-, and $\{4,5\}$-tessellations). In the particular case of Euclidean animals, there are also asymptotic results and computational approximations that give a better understanding of the problem of animal enumeration \cite{ barequet2021improved,barequet2019improved, mansour2021enumeration, shalah2017formulae}. 
In their closing thoughts in \cite{harary1976extremal}, Harary and Harborth raised the question of enumeration of extremal animals with $n$ tiles. Surprisingly, this problem has remained largely unexplored, and it has been completely solved only for $\{4,4\}$-animals \cite{kurz2008counting}. Before now, nothing was known about the enumeration of extremal hyperbolic animals. As a consequence of our main results, for all $\{p,q\}$ we find special sequences of extremal animals that are \textit{unique} extremal animals, in the sense that any animal with the same number of tiles which is distinct up to isometries can't be extremal. 


\vskip.5cm
\subsection{Main Results}

We begin by analyzing a sequence of layered animals, which we use as benchmarks for computing extremal values. These layered animals will attain extremal values, but more significantly they provide context for keeping track of how each subsequent tile affects all of the graph parameters in play.
\vskip.75cm
\begin{rmk}
Note that in a $\{p,q\}$-animal, a vertex can have up to $q$ incident tiles. If it has fewer than $q$ incident tiles, it is necessarily a perimeter vertex. If it has a full set of $q$ incident tiles, it is an interior vertex, and we say that it is \emph{saturated}.
\end{rmk}

\newpage
\begin{defn} \label{defn:completelayers}
We denote by $A_{p,q}(1)$ the animal with only one $p$-gon. We construct $A_{p,q}(k)$ from $A_{p,q}(k-1)$ by adding precisely the tiles needed to saturate all perimeter vertices of $A_{p,q}(k-1)$. The new tiles are called the \emph{$k$-th layer} of the structure. We call $A_{p,q}(k)$ the \emph{complete $k$-layered $\{p,q\}$-animal}.
\end{defn}

\begin{figure}[]
    \centering
    \includegraphics[scale=.6]{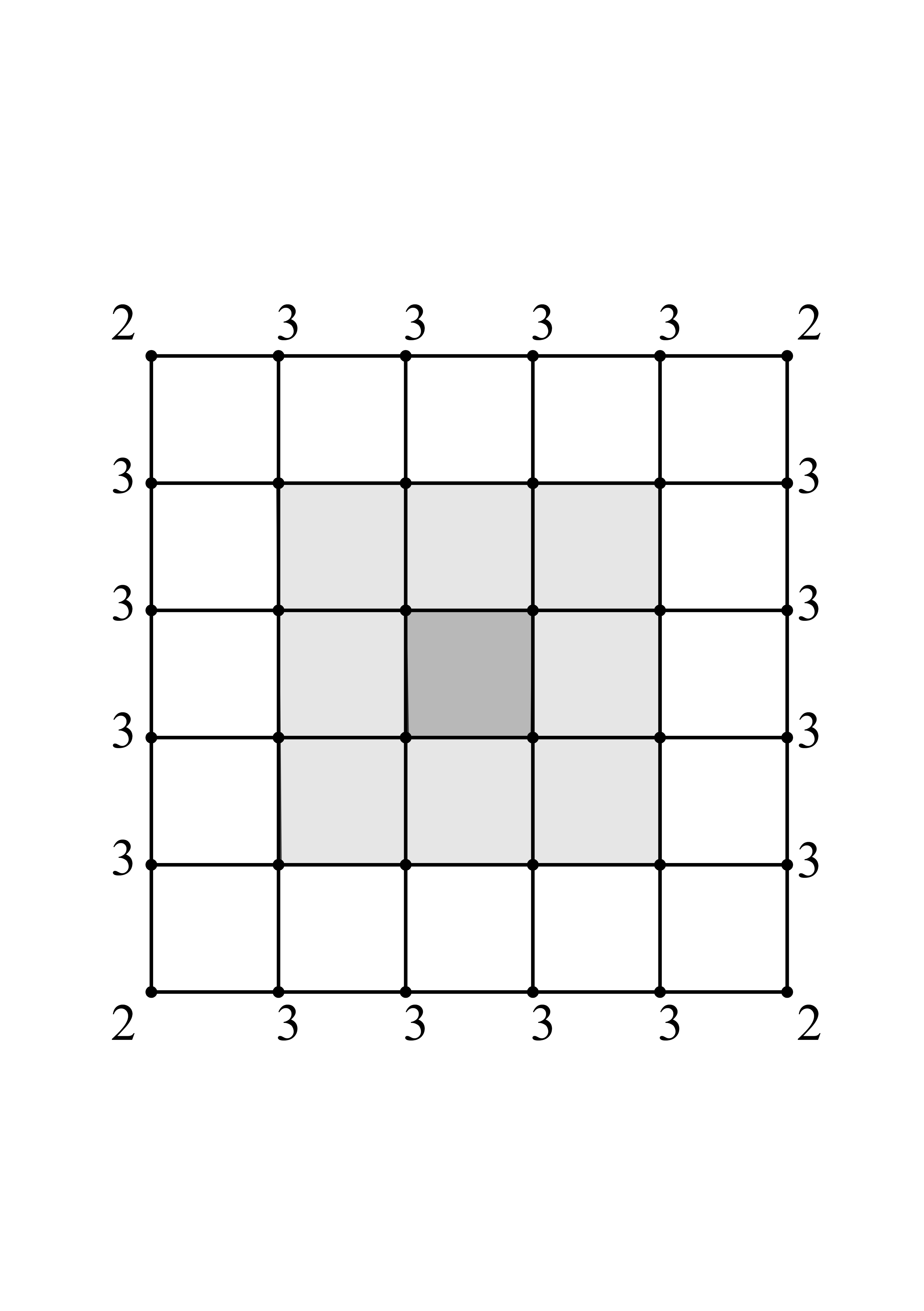}
    \includegraphics[scale=3]{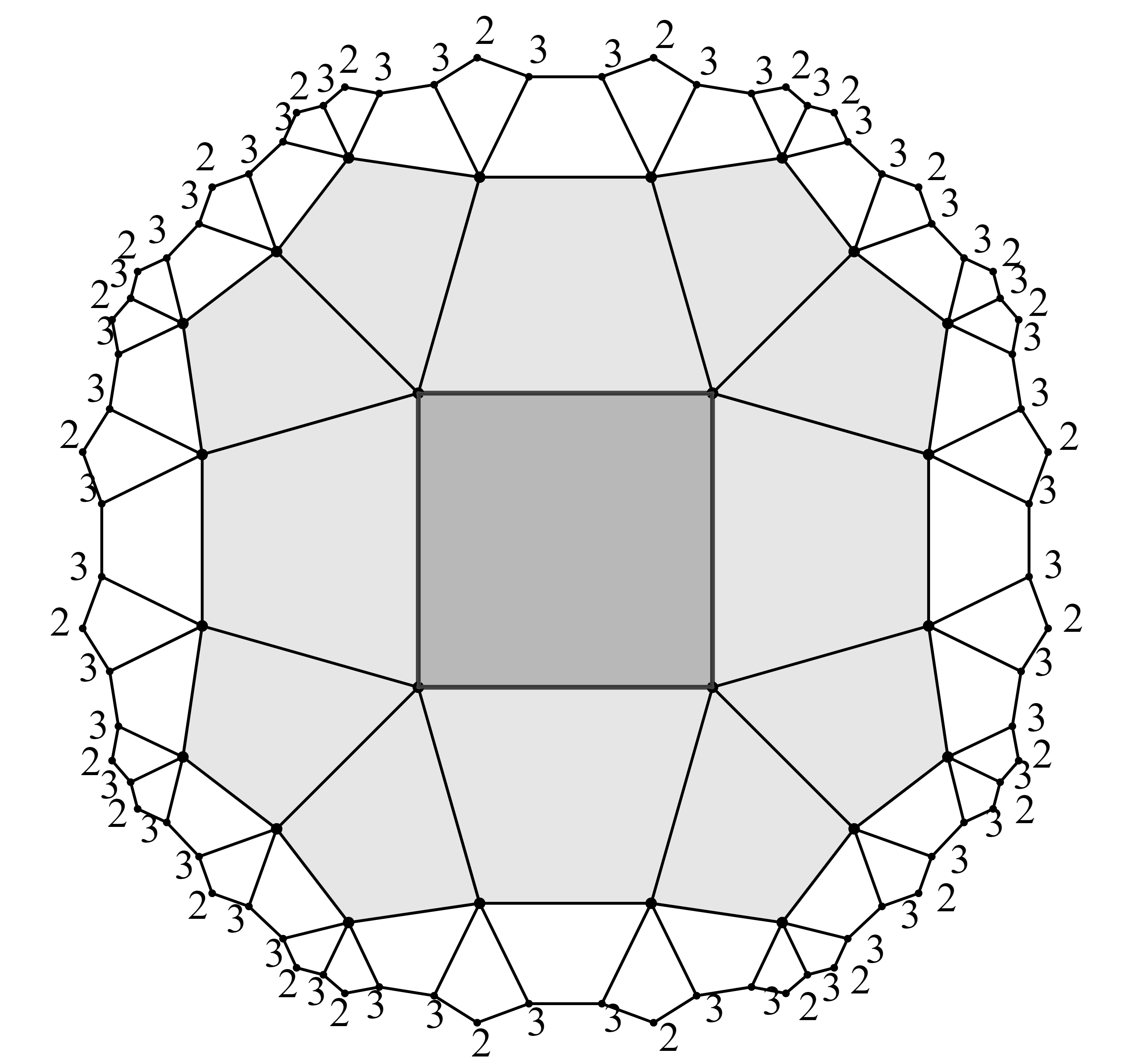}
    \caption{\textit{Left:} $A_{p,q}(3)$ for $\{p,q\}=\{4,4\}$. \textit{Right:}  $A_{p,q}(3)$ for $\{p,q\}=\{4,5\}$. The degree of each perimeter vertex--the number of edges that are incident to the vertex--is written next to it. The case $\{4,5\}$ is hyperbolic, and it is drawn in the Poincaré disk.}
    \label{fig:A3}
\end{figure}


\begin{thm}     \label{thm:layers}
Denote by $\vk(k)$, $\ek(k)$, and $\nk(k)$ the number of vertices, edges, and tiles, respectively, of the complete $k$-layered $\{p,q\}$-animal $A_{p,q}(k)$. Let $t=(p-2)(q-2)-2$ and $\alpha=\frac{t+\sqrt{t^2-4}}{2}$, then 

    \begin{enumerate}
    \item For t=2, we have Euclidean animals, and they satisfy
        \begin{align*}
       \vk(k) &= p k^2,   \\
        \ek(k) & = pk \left( \frac{q}{2}k-\frac{q}{2}+1\right),    \\
        \nk(k) &= 1+ \frac{p(q-2)}{2}k^2 - \frac{p(q-2)}{2}k.
        \end{align*}
    \item For $t>2$, we have hyperbolic animals, and they satisfy
        \begin{align*}
        \vk(k)&=  \frac{p}{t-2} \left( \alpha^k+ \frac{1}{\alpha^k} -2\right),  \\
        \ek(k)    &=   \frac{p}{(\alpha-1)\sqrt{t^2-4}} \left( (\alpha+q-1)\alpha^k+\frac{\alpha q -\alpha+1}{\alpha^k} - q(\alpha+1) \right),    \\
        \nk(k)&= 1 + \frac{p(q-2)}{(\alpha-1)\sqrt{t^2-4}}  \left( \alpha^k + \frac{1}{\alpha^{k-1} } -\alpha-1   \right).  
        \end{align*}
   \end{enumerate}
We also find algebraic expressions in Equation (\ref{eq:allAparams}) for the following graph parameters: the number of interior vertices, perimeter edges, and interior edges of $A_{p,q}(k)$, denoted respectively by $\vk_{int}(k)$, $\ek_1(k)$, $\ek_2(k)$.
\end{thm}

For $t>2$, we get as a corollary that the parameters of $A_{p,q}(k)$ grow exponentially; as expected from the properties of hyperbolic geometry. The exponential growth rate of these parameters is governed by the constant $\alpha>1$.

\begin{cor}
Let $t=(p-2)(q-2)-2\ge 2$ and $\alpha=\frac{t+\sqrt{t^2-4}}{2}$. Then
\[\lim_{k \to \infty}\vk(k+1)/\vk(k)= \alpha,\]
and the same limit is true for the rest of the graph parameters listed in Theorem \ref{thm:layers}.
\end{cor}

Some of the recursive formulas for the complete layered sequence $A_{p,q}(k)$ in Theorem \ref{thm:layers} were independently computed and used as a tool for understanding geometric and combinatorial properties of the $\{p,q\}$-tessellations. However, the extremality of these values was not previously examined. 


For instance, in \cite{higuchi2003isoperimetric}, Higuchi and Shirai use the layered $A_{p,q}(k)$ structures to compute the Cheeger constant of the infinite graph of the $\{p,q\}$-tessellations. Specifically, they give recursive computations for $\vk(k)$ and a value equivalent to the number of tiles in a layer, i.e., $\nk(k)-\nk(k-1)$. Using these values they proved that the Chegeer constant of the $\{p,q\}$-tessellation, defined by
  $$ \inf \left\{ \frac{|E(\partial_v W)|}{\text{vol}(W)} : W \text{ is a finite subgraph} \right\},$$
is equal to the limit of the corresponding values for $A_{p,q}(k)$ as $k\to \infty$. Here, $|E(\partial_v W)|$ is the number of edges connecting $W$ with its complement, and vol$(W)=\sum_{x\in W} deg(x)$.


The $A_{p,q}(k)$ serve as a discrete analogue of balls of radius $k$ centered at $A_{p,q}(1)$, where we consider two faces to be adjacent if their intersection is non-empty. More commonly, proper combinatorial balls with respect to edge-distance in a graph with vertex set $V$ are defined as $\mathcal{B}_k(x)=\{y\in V : d(x,y)\leq k\}$. Then viewing the $\mathcal{B}_k(x)$ as a structure on the dual graph of an animal, the definitions of $A_{p,q}(k)$ and $\mathcal{B}_k(x)$ coincide when $q=3$, since every face in the $k$-th layer is necessarily glued along an edge to a face of $A_{p,q}(k-1)$. However, when $q\ge 4$ there will always be new faces only glued to the previous layer at a vertex, and $\mathcal{B}_k(x)$ will correspond to a sub-animal of $A_{p,q}(k)$. 


In \cite{keller2008geometric}, Keller and Peyerimhoff compute the perimeter of these combinatorial balls $\mathcal{B}_k(x)$. These values are then used to find the \textit{exponential growth} of the graph, defined as,
$$\mu= \limsup_{k\to \infty } \frac{\log \text{vol} (\mathcal{B}_k(x)) }{k}. $$
 They obtain $\mu= \tau + \sqrt{\tau^2-1}$ for the $\{p,q\}$-tessellation, where $\tau=q-\frac{2}{p-2}$. This value is related to our $t=(p-2)(q-2)-2$ by $\tau=\frac{t}{p-2}+2$. Moreover, Theorem \ref{thm:layers} can be stated in terms of the combinatorial curvature at a vertex of a $\{p,q\}$-tessellation, given by $\kappa= 1-\frac{q}{2}+\frac{q}{p} = \frac{2-t}{2p}$. 

In \cite{mertens2017percolation}, Mertens and Moore find recursive formulas for the number of vertices and tiles on the sequence $A_{p,q}(k)$ in order to sample a certain type of random hyperbolic $\{p,q\}$-animals that allows them to give computational estimates of hyperbolic percolation thresholds on $\{p,q\}$-tessellations.

Instead of relying on these independent and related works, we give a complete proof of Theorem \ref{thm:layers} to keep the discussion self-contained. In addition, our analysis is aimed at a more detailed description of the perimeter structure of $A_{p,q}(k)$, which we use to compute the graph parameters of all extremal $n$-tile $\{p,q\}$-animals.


After proving Theorem \ref{thm:layers}, we construct and analyze a sequence of spiral animals with $n$ tiles, of which the complete $k$-layered animals are a subsequence. These $\{p,q\}$-spirals are denoted by $S_{p,q}(n)$ and are constructed in roughly the following way: find the maximum $k$ such that $\nk(k) \leq n$, then attach $n-\nk(k)$ adjacent tiles in one direction along the boundary of $A_{p,q}(k)$, saturating perimeter vertices one a time along the way. The definition of the sequence $S_{p,q}(n)$ is made more precise in Definition \ref{defn:spiral} (see Figure \ref{fig:S37-4-5}). 

Before discussing the spirals in more detail, we pause here for a minute to point out the subtleties in the deluge of notation and bookkeeping we employ to keep track of all the parameters we are interested in. In particular, the same set of graph parameters are discussed in three distinct settings, and we use variations on the notation $v$, $e$, and $n$ for each. First, we have the parameters of $A_{p,q}(k)$, which use the adorned letters presented in Theorem \ref{thm:layers} with the layer value as the input. For the spiral $S_{p,q}(n)$ we use unadorned letters and the tile number as the input, e.g., $\vspir(n) = v(S_{p,q}(n)), \espir(n) = e(S_{p,q}(n))$, etc. For a $\{p,q\}$-animal $A$ or it's underlying planar graph, we use unadorned notation and input the animal's name, e.g., $v(A)$, $e(A)$, etc.


For the spiral $S_{p,q}(n)$, we compute its parameters in terms of the sequence of vertex degrees on the perimeter of the previous layer of the spiral. To accomplish this, it is useful to label the vertices on the perimeter of $A_{p,q}(k)$ by $x_{k,i}$, and to denote by $d_{k,i}$ the degree of $x_{k,i}$. The sequences $d_{k,i}$ are computed recursively at the end of Section 3.

\begin{figure}
    \centering
    \includegraphics[scale=4]{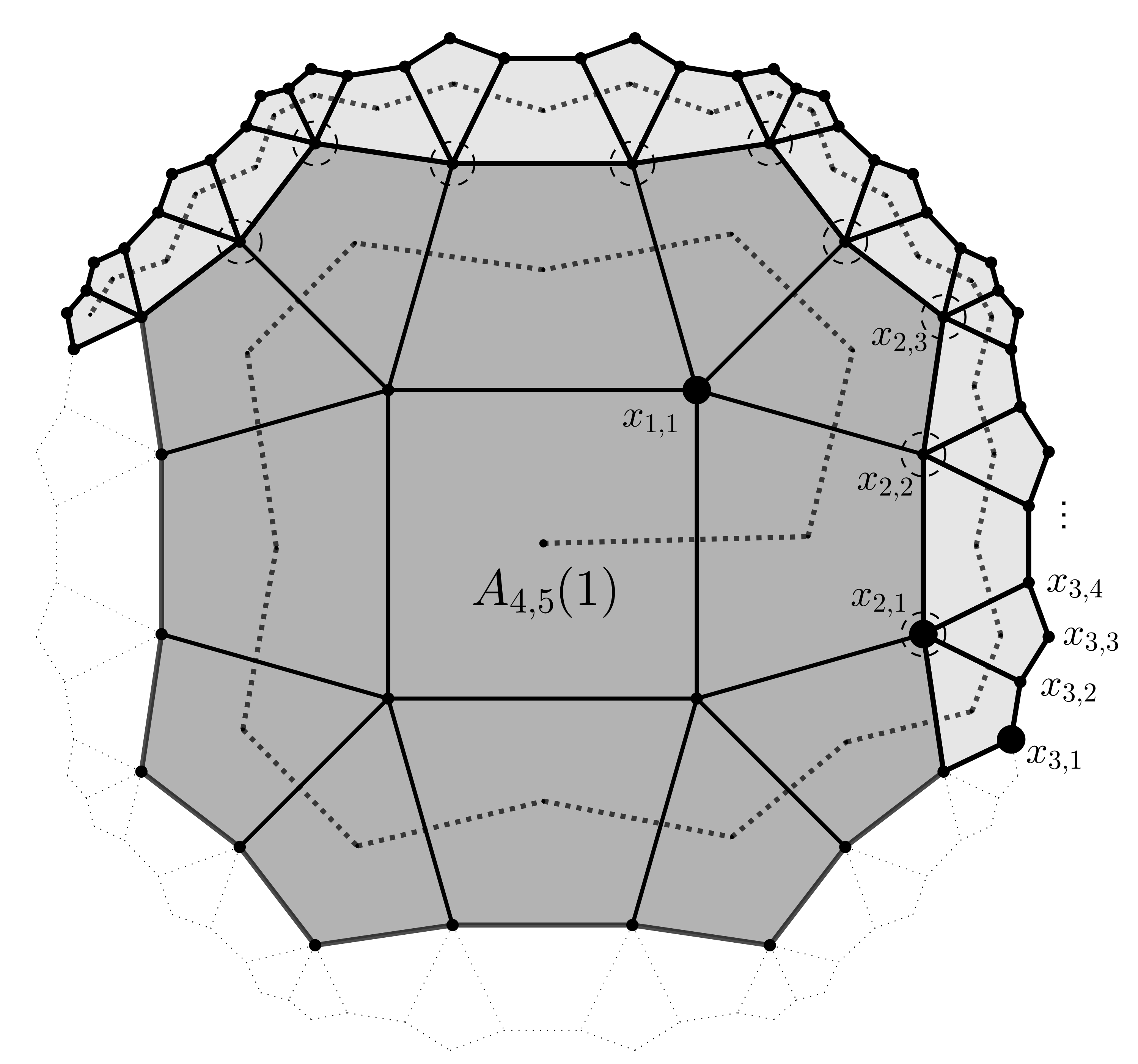}
    \caption{$S_{p,q}(n)$ for $n=37$ and $\{p,q\}=\{4,5\}$. $A_{4,5}(2)$ is shown in dark-gray, the light-gray tiles are those necessary to complete $S_{4,5}(37)$. The growth of the spiral is depicted by the dotted lines. Big-dot points mark the beginning (in the counter-clockwise direction) of the perimeter vertices on each layer; these are denoted by $x_{k,1}$. The number of perimeter vertices of $A_{4,5}(2)$ that are interior vertices of $S_{4,5}(37)$ is $m=9$; these vertices are marked by a dashed circle. See \cite{beautifulanimals} for our interactive applets to explore spiral animals and their graph parameters.}
    \label{fig:S37-4-5}
\end{figure}




\begin{thm}\label{thm:e_1spiral}
Fix $\{p,q\}$ such that $(p-2)(q-2)\ge 4$. Let $S_{p,q}(n)$ be the $\{p,q\}$-spiral with $n$ tiles, and let $k$ be such that $\nk(k)\leq n < \nk(k+1)$. Let $m$ count the number of perimeter vertices of $A_{p,q}(k)$ that are saturated in $S_{p,q}(n)$. Then, 
    \begin{align*}
    \espir(n) &= (p-1)(n-\nk(k))+\ek(k)-m,    \\
    \vspir(n) &= (p-2)(n-\nk(k))+\vk(k)-m.    \\
    \espir_1(n) &= (p-2)(n-\nk(k)) + \ek_1(k)-2m,   
    \end{align*}
Furthermore, $m$ can be computed as follows: If $n-\nk(k)\leq q-d_{k,1}$, then $m=0$. And otherwise $m$ takes the unique value $\ge1$ such that
    \begin{align} \label{eqn:formula-m}
        \sum_{i=1}^m (q-d_{k,i}) \leq n-\nk(k)-1 < \sum_{i=1}^{m+1} (q-d_{k,i}).        
    \end{align}
\end{thm}






The next theorem states that the spirals in fact attain extremal values in the graph parameters listed above. It is natural to imagine that this should be the case, as the spiral configuration intuitively has the effect of wrapping up as many vertices and edges in the interior as possible, where they can be shared by the most tiles. Thus, interior parameters are maximized while perimeter and overall counts are minimized. 


\begin{thm}\label{thm:extremalspiral}
Fix $\{p,q\}$ such that $(p-2)(q-2)\ge 4$, and let $A$ be a $\{p,q\}$-animal with $n$ tiles. Then $e_2(A)$ and $v_{int}(A)$ are maximized, and $e_1(A)$, $e(A)$, and $v(A)$ are minimized when $A=S_{p,q}(n)$.
\end{thm}

\textit{Spiral-like} animals have also been used by Buchholz and de Launey \cite{buchholz2009edge} for exploring the extremal combinatorial problem of edge-minimization for families of $\{p,q\}$-animals that do not tessellate the Euclidean plane. They use spiral-like arrangements to obtain asymptotic results for extremal values on the number of internal edges. In that setting, it remains open as to whether these spiral-like arrangements are actually extremal, that is, if they attain and provide the exact minimal optimal values for the number of internal edges.

Theorem \ref{thm:extremalspiral} shows that the minimum perimeter attainable by a $\{p,q\}$-animal with $n$ tiles, which we denote by $\mathcal{P}_{p,q}(n)$, is equal to $e_1(n)$. Plugging the Euclidean $\{p,q\}$ values into Theorem \ref{thm:e_1spiral}, we recover Harary and Harborth's formulas for extremal values of animals on regular tessellations of the Euclidean plane \cite{harary1976extremal}.

\begin{cor}  \label{cor:Harary}  
For $\{p,q\}=\{3,6\}$, $\{4,4\}$, and $\{6,3\}$, we have
    \begin{align*}
        \mathcal{P}_{3,6}(n) &= 2 \left\lceil \frac12( n+\sqrt{6n}) \right\rceil - n,  \\    
        \mathcal{P}_{4,4}(n) &= 2 \left\lceil 2\sqrt{n} \right\rceil,                  \\ 
        \mathcal{P}_{6,3}(n) &=  2 \left\lceil \sqrt{12n-3} \right\rceil.  
    \end{align*}
\end{cor}

While the Euclidean cases immediately yield closed formulas, in the hyperbolic case it remains open to find a closed formula for $\mathcal{P}_{p,q}$ in terms of $n$. 




Finally, we address the problem of enumerating extremal $n$-tile animals for a fixed $n$, exhibiting several sequences which give unique extremal animals up to isometries. The question of counting extremal animals was first posed by Harary and Harborth in \cite{harary1976extremal}. In the $\{4,4\}$ case \cite{kurz2008counting}, Kurz proved that squares and pronic rectangles, i.e. polyominoes with $l^2$ tiles and rectangles with $l(l+1)$ tiles for $l\geq 1$, respectively, are unique extremal animals up to isometries. Here, we find sequences of $\{p,q\}$-animals which generalize this result for all $\{p,q\}$ pairs in Theorems \ref{thm:uniq_layered} and \ref{thm:pronic}.


The rest of the paper is structured as follows: In Section 2, we set up notation and discuss graph theoretic formulae we will need later on. In Section 3, we analyze the layered animals $A_{p,q}(k)$, we deduce formulas for their parameters to prove Theorem \ref{thm:layers}, and we write the substitution rules necessary to construct their perimeter sequence of vertex degrees $d_{k,i}$. We conclude Section 3 with detailed examples that highlight the difference between the Euclidean and hyperbolic cases. Then, in Section 4, we prove Theorem \ref{thm:e_1spiral} by finding algebraic expressions for the parameters of $S_{p,q}(n)$ in terms of the sequences $d_{k,i}$. In Section 5, we provide the proof of Theorem \ref{thm:extremalspiral} which states that spirals are indeed extremal. In Section 6, we prove Theorems \ref{thm:uniq_layered} and \ref{thm:pronic} regarding unique extremal animals. And finally, in Section 7, we state open problems and give concluding remarks. 

\section{Definitions and Preliminary Results}



To begin, we restrict ourselves to animals with no holes, where a hole is defined to be a finite connected component of the complement. We then extend results on extremality to animals with holes as the last piece in proving Theorem \ref{thm:extremalspiral}. 

Consider a $\{p,q\}$-animal $A$ with no holes. The edges and vertices of the $p$-gons of $A$ define a planar graph in which all vertices have degree at least 2 and at most $q$, and its bounded faces are all regular $p$-gons. Here, and throughout this paper we use the term faces only to refer to the \emph{bounded} faces of a planar graph. So the number of tiles of $A$ is equal to the number of faces of its underlying planar graph, which we also refer to simply as $A$, conflating the notation for convenience.

There is also a dual graph $G'$ associated with any planar graph $G$, constructed by placing a vertex at each face and connecting two vertices if and only if their corresponding faces share an edge. For an animal $A$, the dual graph $A'$ may not be an animal, as some edges may not be contained in any faces. Thus we define a larger class of planar graphs to work in, which contains all graphs defined by animals with no holes.

\begin{defn}
We call $G$ a $\{p,q\}$-graph if $G$ is a planar graph in which each vertex of $G$ has degree at most $q$, and each face of $G$ has exactly $p$ edges. 
\end{defn}

Then a $\{p,q\}$-animal with no holes is necessarily a $\{p,q\}$-graph. However, not every $\{p,q\}$-graph is the graph of an animal or even a subgraph of the regular $\{p,q\}$-tessellation. For instance, we can build a  rooted $\{p,q\}$-tree starting with a single vertex, the root $R$, which has $q$ adjacent vertices. Then we give each of $R$'s neighbors $q-1$ distinct additional neighbors, and then each of those new vertices get $q-1$ distinct additional neighbors, and so on. After more than $p/2$ iterations of this, such a graph can no longer live in the $\{p,q\}$-tessellation, as these strings starting from $R$ would necessarily form $p$-gons.
 
For the rest of this paper, we always assume that $(p-2)(q-2)\ge4$ for the pair $\{p,q\}$. We denote the total number of vertices, edges and faces of a fixed graph $G$ in the same manner as for a $\{p,q\}$-animal, using $v(G)$, $e(G)$ and $n(G)$, respectively. Furthermore, we partition the vertices and edges by the number of incident edges and faces, respectively:
\begin{align*}
    v_i(G) &= \# \text{ vertices of degree } i, \text{ for } 1\le i \le q-1,\\
    v_q(G) &= \# \text{ vertices of degree $q$ incident to only $q-1$ faces,}\\
    v_{int}(G) &= \# \text{ interior vertices, i.e., vertices of degree $q$ incident to $q$ faces,}\\
    e_i(G) &= \# \text{ edges incident to $i$ faces, for } i = 0,1,2.
\end{align*}

All non-interior vertices are called \emph{perimeter} vertices, and these are partitioned by the $v_i$ according to their degrees. 

When a fixed $G$ is understood, we write simply $v$, $e$, $n$, etc, and for its dual graph $G'$ we write $v'$, $e'$, $n'$, etc. It is easily seen that when $G$ is a $\{p,q\}$-graph, $p$ and $q$ switch roles for $G'$, which is a $\{q,p\}$-graph. In general, $v'$, $e'$ and $n'$ can be counted directly by various parameters of $G$. We know by definition that $v' = n$; the edges of $G'$ correspond to edges of $G$ which are incident to two faces of $G$, and hence $e' = e_2$; and a face of $G'$ corresponds to an interior vertex of $G$, so $n' = v_{int}$.  We further partition the edges of the dual graph and count its connected components:

\begin{align*}
    e_{2,i}(G) &= \# \text{ edges in $G'$ incident to $i$ faces, for } i = 0,1,2,\\[4pt]
    c'(G) &= \# \text{ connected components of $G'$.}\\
\end{align*}

\begin{figure}
    \centering
    \includegraphics[scale=3]{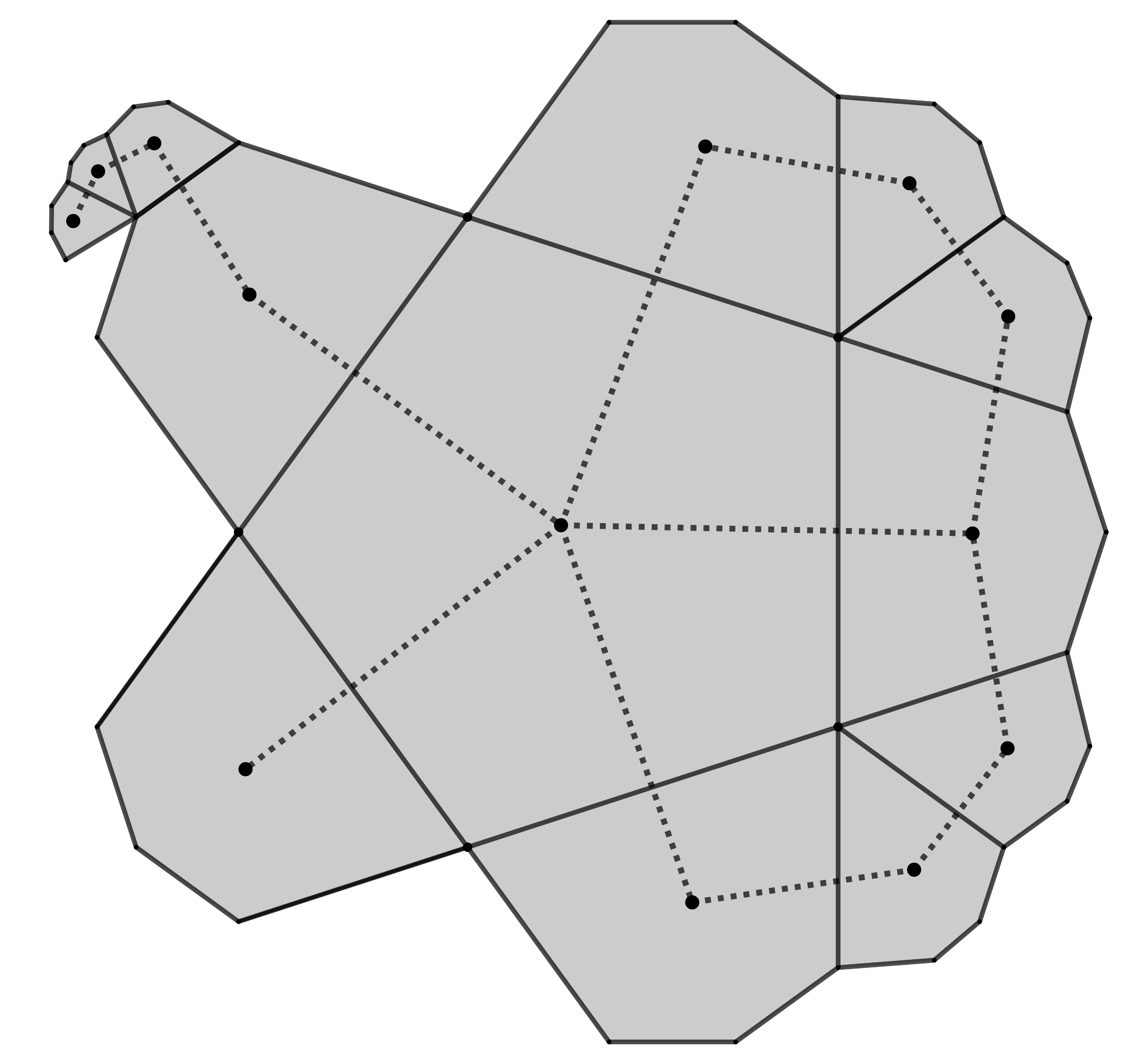}
    \caption{A $\{5,5\}$-animal and its dual. Its parameters are $n=13$, $e=50$, $e_1=36$, $e_2=14$, $v=38$, $v_{int}=2$, $e_{2,0}=5$, $e_{2,1}=8$ and $e_{2,2}=2$. Any resemblance to a real animal is pure coincidence; no animals were harmed for this research.}
\end{figure}

When $G$ is an animal, $c'$ is always 1 by definition. But here we will obtain more general bounds for all $\{p,q\}$-graphs. Using Euler's formula for $G'$, and also counting its edges with respect to incident faces, we have that: 

    \begin{align*}
        n-e_2+v_{int}&=c',   \\
        q\cdot v_{int} &= e_{2,1}+2e_{2,2} =e_2 -e_{2,0}+e_{2,2}.  
    \end{align*}
Combining these equations to eliminate $v_{int}$ yields
    \begin{align}
        e_2 = \frac{q(n-c') -e_{2,0} +e_{2,2}}{q-1}. \label{eqn:interior-edges}
    \end{align}
This formula will allow us to use induction to prove that $e_2(G)\leq e_2(S_{p,q}(n))$ for any $\{p,q\}$-graph with exactly $n$ faces by translating the inequality to the dual graphs.

We will also need the following lemma which asserts that the number of tiles decreases when we perform the dual operation. Observe that if a $\{p,q\}$-graph is a \emph{tree}, that is a connected graph with no faces, then $v_{int}=0$ and the dual graph of that component is the empty graph. Thus, it is sufficient to only consider components with faces. 

\begin{lem} \label{lem:v_intbound}
Let $G$ be a $\{p,q\}$-graph with $q\geq 4$ such that every connected component has at least one face. Then $v_{int}\leq n-1$. Moreover, if $q=3$ then $v_{int} \leq 2n-2$, and if $q\geq 6$ then $v_{int}< \frac{n}{2}$.
\end{lem}

\begin{proof}
The Handshake Lemma and Euler's formula for $G$ are:    
    \begin{align*}
        q \cdot v_{int} + \sum_{i=1}^q i\cdot v_i &= 2e, \\ 
        v - e +n &= c.       
    \end{align*}

Using these equations to eliminate $e$ we obtain:
    $$ (q-2)v_{int} + \sum_{i=1}^q (i-2)\cdot v_i = 2n-2c.  $$

The sum of $(i-2)\cdot v_i$ can be negative if $v_1 > 0$, and to avoid this we iteratively prune all degree 1 vertices, one at a time, until there are none left. Cycles in a graph cannot be deleted via this process, therefore since each component is assumed to have at least one face, the resulting graph has the same number of faces, interior vertices, and connected components as $G$. Let $\bar{v}_i$ denote the corresponding parameters in the reduced graph, with $\bar{v}_1=0$. Then

    $$ (q-2)v_{int} \leq  (q-2)v_{int} + \sum_{i=2}^q (i-2)\cdot \bar{v}_i = 2n-2c \leq  2n -2.$$ 
This gives the desired bounds, choosing the appropriate $q$ in each case.
\end{proof}



\section{Complete $k$-layered $\{p,q\}$-animals}  \label{section:layers}

In this section, we define and study the main properties of the complete $k$-layered $\{p,q\}$-animals $A_{p,q}(k)$. Recall from Definition \ref{defn:completelayers} that $A_{p,q}(1)$ is the regular $p$-gon, and then $A_{p,q}(k)$ is constructed inductively by attaching all allowable tiles to the perimeter of $A_{p,q}(k-1)$. The graph parameters of $A_{p,q}(k)$ are denoted by $\vk(k)$, $\ek(k)$, $\nk(k)$, etc. 


As each new layer is added, these parameters can be tracked by characterizing the geometry of how individual tiles are attached. When adding a tile $T$ to an animal $A$, we use the notation $\deg_{A}(T)$ to denote the number of edges $T$ shares with $A$. This is, in fact, the degree of the vertex representing $T$ in the dual graph of $A\cup T$. More informally, we describe a tile being added to a given animal as being \emph{$\epsilon$-glued} if $\deg_{A}(T)=\epsilon$. In constructing $A_{p,q}(k)$ from $A_{p,q}(k-1)$, we consider each tile individually and examine $\deg_{A_{p,q}(k-1)}(T)$ for every tile $T$ in the $k$-th layer. We will use this parameter to classify a layer's perimeter vertices by degree, and then find recursive formulas counting the number of each degree in a layer.

One key observation is that the intersection of a tile in the $k$-th layer with $A_{p,q}(k-1)$ is a connected path, which may be only a single vertex, in which case the tile is \emph{0-glued}. Then a given perimeter edge gets covered with a single tile, and we must consider if that tile can cover any adjacent edges on the perimeter as well. For instance, a tile that is 2-glued will completely cover one shared vertex with $A_{p,q}(k-1)$, which is possible if and only if that vertex already has degree $q$ in $A_{p,q}(k-1)$. Otherwise 0-glued tiles must be filled in next to a 1-glued tile until the degree of the common perimeter vertex is $q$ (see Figure \ref{fig:A3-recursion}).


\begin{figure}   
    \centering
   \includegraphics[scale=4.5]{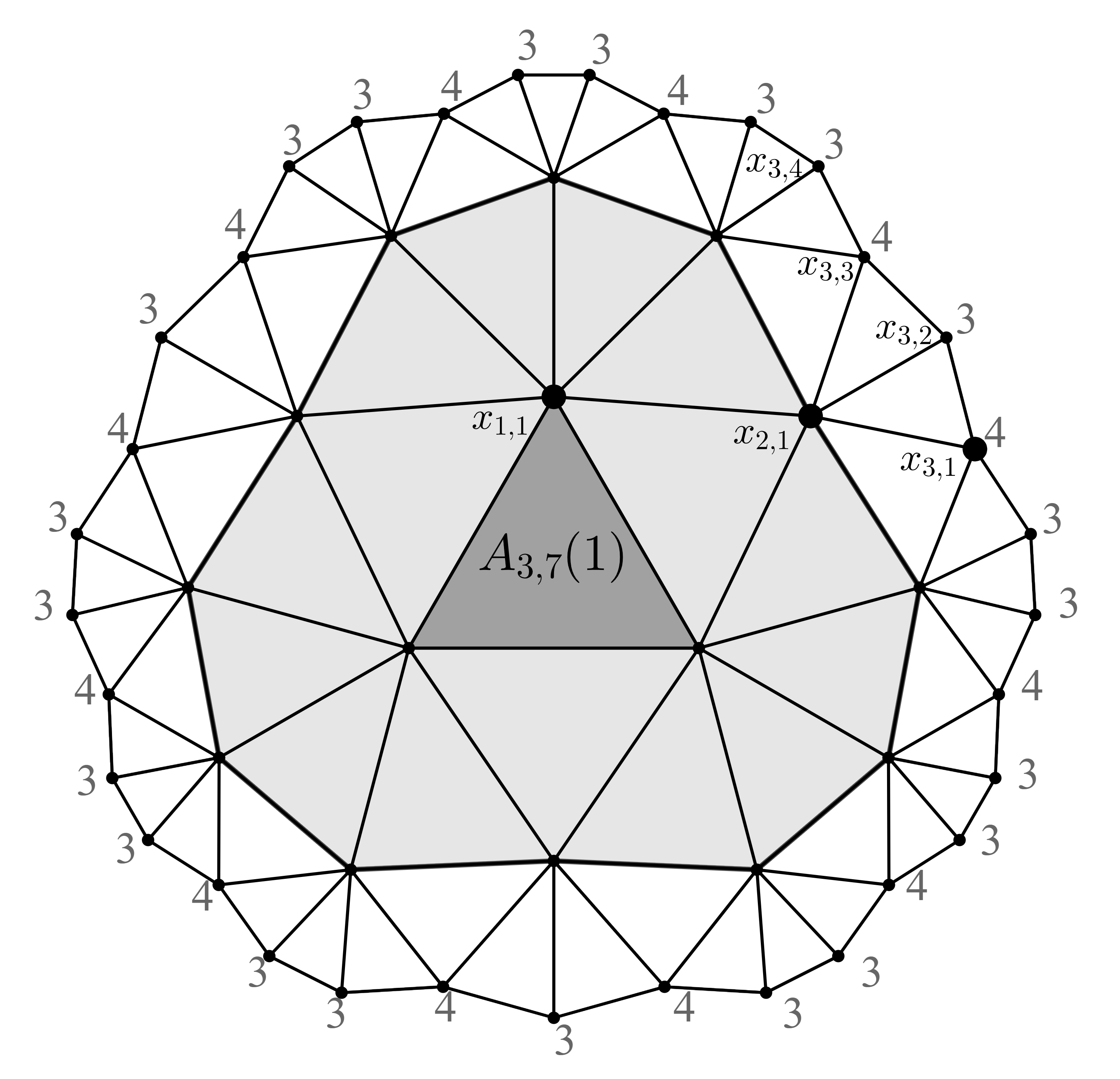}
    \caption{Here we depict $A_{p,q}(k)$ for $\{p,q\}=\{3,7\}$ and $k=3$. Big-dot points mark the beginning of the perimeter vertices on each layer, these are denoted by $x_{k,1}$. The degree of each perimeter vertex of $A_{3,7}(3)$ is written next to it; this is the sequence $d_{3,i}$, starting at $x_{3,1}$ in the counter-clockwise direction.
    See Figure \ref{fig:A3} for the cases $\{4,4\}$ and $\{4,5\}$.} \label{fig:A3-3-7}.
\vspace{-.5cm}
\end{figure}
\begin{rmk}\label{rmk:vdeg}
If $p=3$, all perimeter vertices of $A_{p,q}(k)$ are degree 3 or 4, for $k\ge 2$. If $p\ge4$, all perimeter vertices of $A_{p,q}(k)$ are degree 2 or 3. 
\end{rmk}

\begin{rmk}\label{rmk:e-gluing}
If $q=3$, then all tiles in $A_{p,q}(k)$ are either 1-glued or 2-glued to $A_{p,q}(k-1)$, with one 2-glued tile for every vertex of degree 3 in $A_{p,q}(k-1)$. If $q\ge4$, every tile is either 0-glued or 1-glued.
\end{rmk}

Inherent in the proof of the recursions in Proposition \ref{prop:recurrence} below is an inductive argument confirming these remarks, which are inextricably linked with each other. For instance, assuming Remark \ref{rmk:vdeg}, if $q\ge4$ there can never be perimeter vertices of degree $q$, because degree 4 vertices only exist for $p=3$, and $p=3$ requires $q\ge6$. Consequently, there cannot be $\epsilon$-glued tiles for $\epsilon\ge 2$ when $q\ge 4$. Conversely, if $p=3$, then $q\ge6$ and every tile is 0-glued or 1-glued, and the ensuing geometry can be used to show that all perimeter vertices have degree 3 or 4.

\begin{prop}        \label{prop:recurrence}
Recall that perimeter vertices of $A_{p,q}(k)$ are counted by $\vk_i(k)$ according to their degree $i$. So $\vk_2(k)$, $\vk_3(k)$ and $\vk_4(k)$ are the number of perimeter vertices of $A_{p,q}(k)$ of degree 2, 3, and 4, respectively. Then, for $p=3$ and $k\geq3$,
    \begin{align*}
    \vk_3(k) &= (q-5)\vk_3(k-1) + (q-6)\vk_4(k-1),    \\
    \vk_4(k) &=  \vk_3(k-1) + \vk_4(k-1). 
    \end{align*}
    For $p\geq 4$ and $k\geq2$, 
    \begin{align*}
    \vk_2(k) &=\left((p-3)(q-2)-1 \right)  \vk_2(k-1) + \left((p-3)(q-3)-1 \right)  \vk_3(k-1),     \\
    \vk_3(k) &= (q-2) \vk_2(k-1) + (q-3) \vk_3(k-1).
    \end{align*}
\end{prop}

\begin{proof}
Suppose $p=3$. Then $q\ge6$ and a perimeter vertex $x$ of $A_{p,q}(k-1)$ of degree $d = 3,4$ belongs to $d-1$ tiles of $A_{p,q}(k-1)$. To construct $A_{p,q}(k)$, we attach $q-d+1$ tiles to $x$. To avoid overlap, we associate to $x$ only the first $q-d$ tiles in the counterclockwise direction. Note that $q-d\ge 2$. The first tile is always 1-glued to $A_{p,q}(k-1)$, while the remaining $q-d-1\ge 1$ are 0-glued at a common vertex. So the first tile contributes one perimeter vertex of $A_{p,q}(k)$ of degree 4, and the next $q-d-2$ tiles each contribute one vertex of degree 3. The last tile associated to $x$ does not add to this count, as it shares one perimeter vertex with the previous tile associated to $x$, and the other with a 1-glued tile whose perimeter vertex is counted by the next vertex of $A_{p,q}(k)$. In summary, all perimeter vertices of $A_{p,q}(k-1)$ of degree 3 contribute in total to $(q-5)\vk_3(k-1)$ perimeter vertices of $A_{p,q}(k)$ of degree 3 and $\vk_3(k-1)$ vertices of degree 4. Similarly, perimeter vertices of $A_{p,q}(k-1)$ of degree 4 contributes to $(q-6)\vk_4(k-1)$ perimeter vertices of $A_{p,q}(k)$ of degree 3 and $\vk_4(k-1)$ vertices of degree 4. By adding these two contributions, we arrive at the recurrence formulas for $p=3$.

\begin{figure}[h]  
    \centering
   \includegraphics[scale=5]{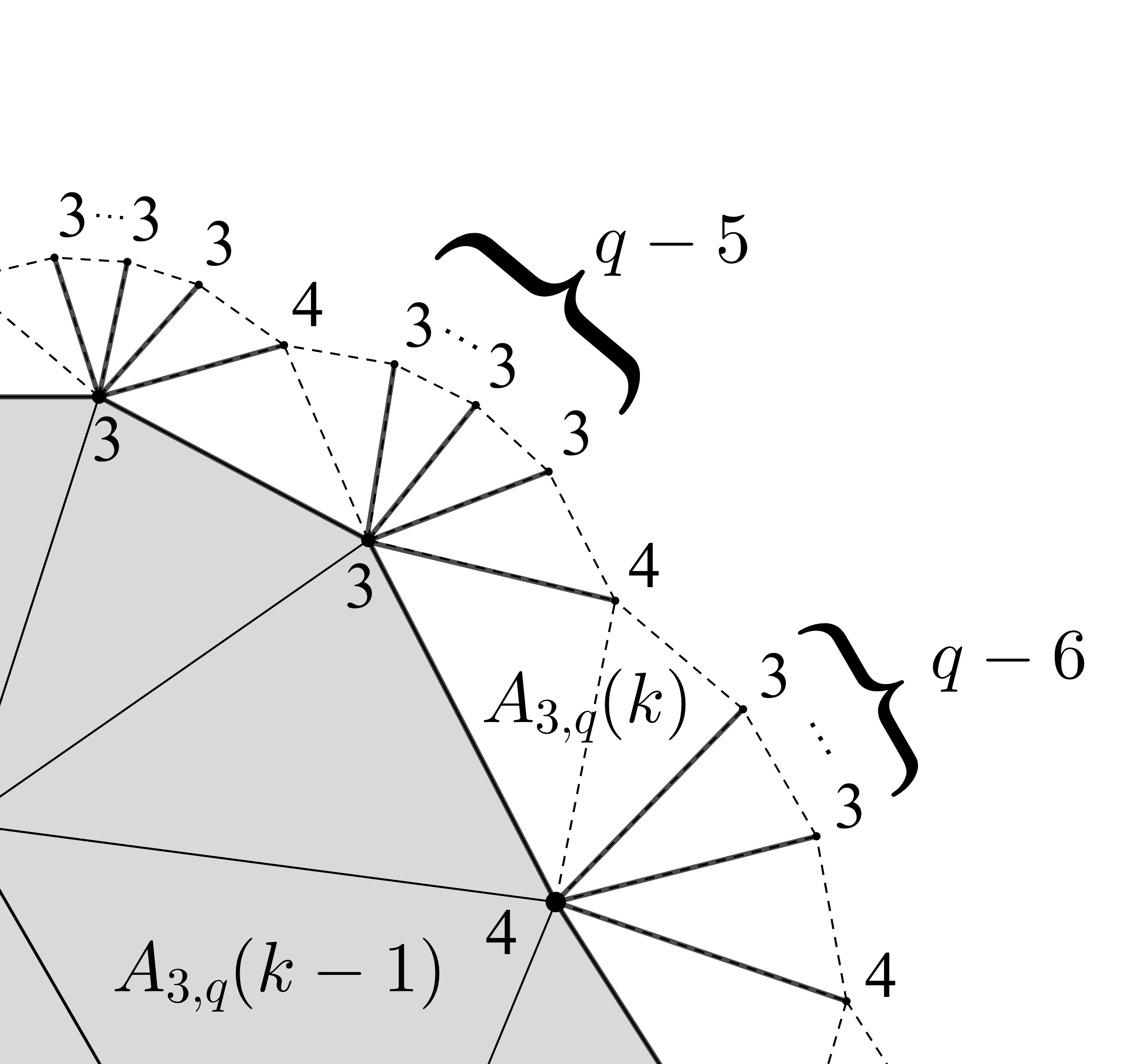}
   \includegraphics[scale=0.045]{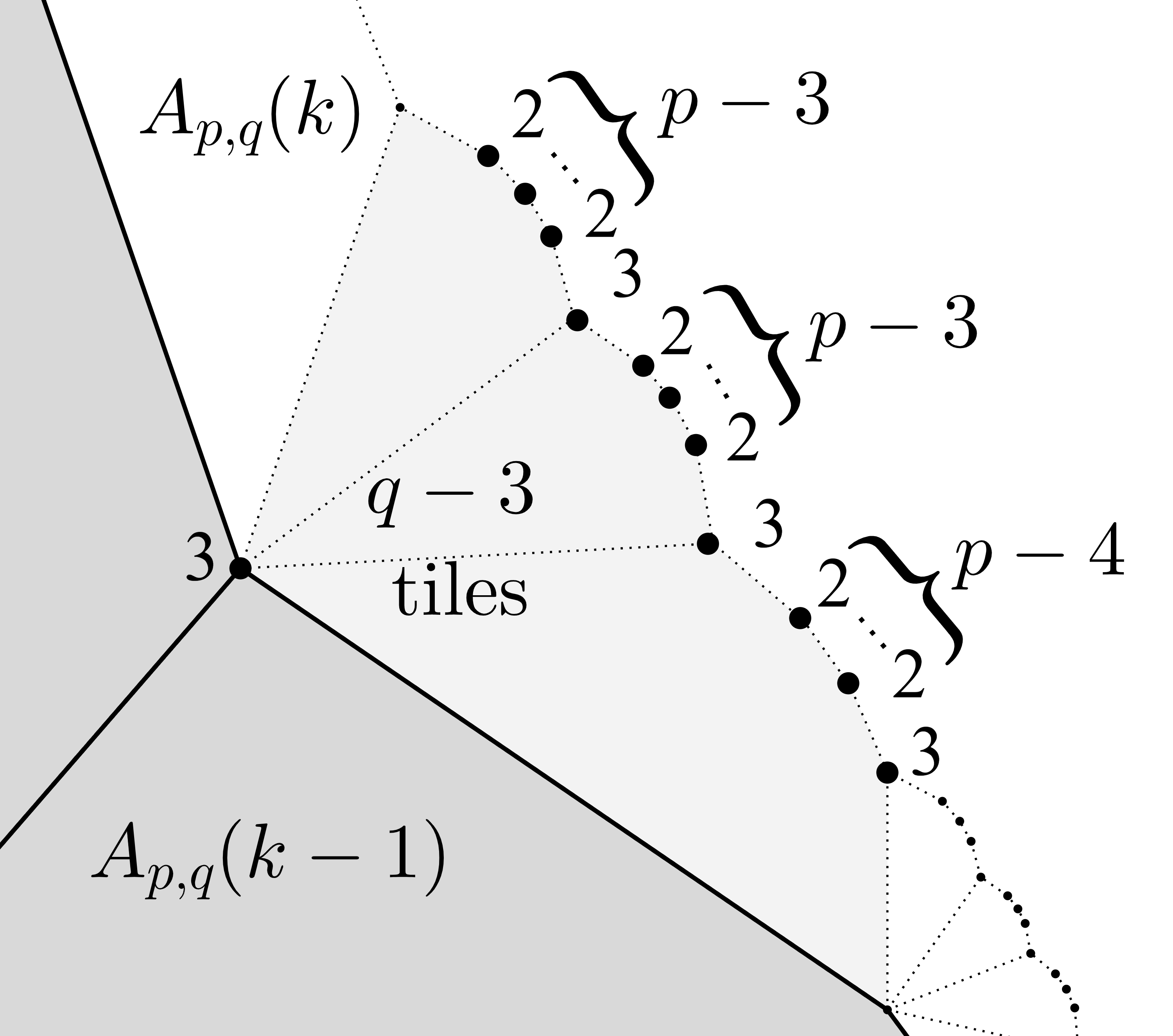}
    \caption{\textit{Left:} When $p=3$, each vertex of degree $d=3$ or $4$ in $A_{p,q}(k-1)$, contributes to a vertex of degree 4 in $A_{p,q}(k)$ and $q-d-2$ vertices of degree $= 3$. \textit{Right:} When $p\geq4$, each vertex of degree $3$ in $A_{p,q}(k-1)$, contributes to $q-3$ vertices of degree $=3$ in $A_{p,q}(k)$ and  $(q-3)(p-3)-1$ vertices of degree $=2$.}  \label{fig:A3-recursion}
\end{figure}

The case for $p\geq 4$ is analogous. Each perimeter vertex of degree $d$ in $A_{p,q}(k-1)$ has $q-d+1$ tiles of $A_{p,q}(k)$ attached to it, but we consider only the first $q-d$ tiles, to avoid overlapping. Note that the first tile is always 1-glued to $A_{p,q}(k-1)$, and the rest are 0-glued to a common vertex. So the first has $p-2$ perimeter vertices, and the rest have $p-1$, but we subtract the last one in a counterclockwise direction from each to avoid overlap. Then each of the $q-d$ tiles contributes one perimeter vertex of $A_{p,q}(k)$ of degree 3, which is incident to the adjacent tile in the clockwise direction, and $p-3$ vertices of degree 2, with the exception of the first tile which has one vertex of degree 3 and $p-4$ of degree 2. In summary, a perimeter vertex of $A_{p,q}(k-1)$ of degree $d$ contributes $q-d$ vertices of degree 3 to $A_{p,q}(k)$ and $(q-d)(p-3)-1$ vertices of degree 2. This gives the recurrence relations for $p\geq4$.
\end{proof}

\begin{lem} \label{lem:recurrence-vert}
Let $t=(p-2)(q-2)-2$ and $\alpha=\frac{t+\sqrt{t^2-4}}{2}$, $\beta=\frac{t-\sqrt{t^2-4}}{2}$. The solution to the recurrence relations of Proposition \ref{prop:recurrence} are:
    \begin{enumerate}
        \item When $t=2$ we have the Euclidean cases:
            \begin{align*}
            \text{For } \{p,q\}= & \,\{3,6\}, \, k\geq2,  & \text{For } \{p,q\}=&\, \{4,4\},       &  \text{For }\{p,q\}=&\,\{6,3\},   \\
            \vk_3(k) &= 6,             &   \vk_2(k)&= 4,          &  \vk_2(k)&= 6k,\\
            \vk_4(k) &= 3+6(k-2),     &   \vk_3(k) &=8(k-1) ,     &  \vk_3(k)&= 6(k-1).
            \end{align*}                
        \item For $p=3$, $q>6$ and $k\geq 2$,
            \begin{align*}
            \vk_3(k) &=\frac{3}{\alpha-\beta} \left( (\alpha-1)(q-3-\beta)\alpha^{k-2} + (1-\beta)(q-3-\alpha)\beta^{k-2}  \right),   \\
            \vk_4(k) &=   \frac{3}{\alpha-\beta} \left( (q-3-\beta)\alpha^{k-2} - (q-3-\alpha)\beta^{k-2}  \right).  
            \end{align*}
        \item For $p\geq 4$ and  $(p-2)(q-2)>4,$
            \begin{align*}
            \vk_2(k) &=\frac{p}{\alpha-\beta} \left( -(q-3-\alpha)\alpha^{k-1} + (q-3-\beta)\beta^{k-1}  \right),   \\
            \vk_3(k) &=   \frac{p}{\alpha-\beta} \left( (q-2)\alpha^{k-1} - (q-2)\beta^{k-1}  \right).  
            \end{align*}        
    \end{enumerate}
\end{lem}

\begin{proof} This is the result of solving a linear system of $2\times2$ recurrence relations; the associated matrix in both cases has determinant 1 and trace $t=(p-2)(q-2)-2$. The characteristic polynomial is then $x^2-tx+1$, and its roots are $\alpha$ and $\beta$. 

When $t=2$, we have $\alpha=\beta=1$, and the growth is linear. When $t>2$, we get exponential growth in terms of $\alpha$ and $\beta$. The initial conditions are $\vk_3(2)=3(q-4)$, $\vk_4(2)=3$ for $p=3$, and $\vk_2(1)=p$, $\vk_3(1)=0$ for $p\geq4$. The diagonalization matrix is
    \begin{align*}
        M =  \begin{pmatrix}
                1-\alpha & 1-\beta \\
                -1 & -1 
            \end{pmatrix}
            \text { for } p=3, \text{ or } 
            M=\begin{pmatrix}
                \frac{q-3-\alpha}{q-2} & \frac{q-3-\beta}{q-2} \\
                -1 & -1 
            \end{pmatrix} \text { for } p\geq 4.
    \end{align*}
The $k$th element of the sequences can be computed from 
    \begin{align*}
        \begin{pmatrix} 
                \vk_3(k) \\
                \vk_4(k)
        \end{pmatrix} = M \begin{pmatrix}
                \alpha^{k-2} & 0 \\
                0 & \beta^{k-2} 
            \end{pmatrix} M^{-1}
        \begin{pmatrix} 
                \vk_3(2) \\
                \vk_4(2)
        \end{pmatrix}, \text{ for $p=3$,}    
    \end{align*}

    \begin{align*}
        \begin{pmatrix} 
            \vk_2(k) \\
            \vk_3(k)
        \end{pmatrix} = M \begin{pmatrix}
                \alpha^{k-1} & 0 \\
                0 & \beta^{k-1} 
            \end{pmatrix} M^{-1}
        \begin{pmatrix} 
                \vk_2(1) \\
                \vk_3(1)
        \end{pmatrix}, \text{ for } p\geq 4.
    \end{align*}
From these we obtain the desired expressions.
\end{proof}

\begin{proof}[Proof of Theorem \ref{thm:layers}] 
Firstly, we work the case $p\geq 4$. The vertices of $A_{p,q}(k)$ consists of its perimeter plus those of $A_{p,q}(k-1)$, that is, $\vk(k)=\vk(k-1)+\vk_2(k)+\vk_3(k)$, then 
    $$\vk(k) = \sum_{j=1}^k \vk_2(j) + \vk_3(j).$$    

As the perimeter of $A_{p,q}(k)$ is a cycle, then $\ek_1(k)$ is equal to the number of vertices on the perimeter. That is, $\ek_1(k)=\vk_2(k)+\vk_3(k)$. Next, $\ek_2(k)$ consists of the interior and perimeter edges of $A_{p,q}(k-1)$, plus the interior edges of $A_{p,q}(k) \setminus A_{p,q}(k-1)$, which also correspond to the edges that connect perimeter vertices of $A_{p,q}(k-1)$ and $A_{p,q}(k)$. The latter are precisely $q-d$ for each vertex of degree $d$ in the perimeter of $A_{p,q}(k-1)$. Then $e_{2}(k)=\ek_2(k-1)+\ek_1(k-1)+(q-2)\vk_2(k-1) + (q-3)\vk_3(k-1)=\ek_2(k-1)+(q-1)\vk_2(k-1)+(q-2)\vk_3(k-1)$. Now, recalling that $\ek_2(1)=0$, we get, 
$$\ek_2(k)= (q-1)\sum_{j=1}^{k-1}\vk_2(j) + (q-2)\sum_{j=1}^{k-1}\vk_3(j).$$
Finally, $\ek(k)=\ek_1(k)+\ek_2(k)$ and $\nk(k)=1-\vk(k)+\ek(k)$. All the previous, together with the expressions for $\vk_2(k)$ and $\vk_3(k)$, give us the formulas for Theorem \ref{thm:layers}. We just have to notice that $\beta=\alpha^{-1}$, $\alpha+\beta=t$, and $\alpha-\beta=\sqrt{t^2-4}$.

The case for $p=3$ is analogous. The only subtle difference is that 
$$\ek_2(k)= (q-2)\sum_{j=1}^{k-1}\vk_3(j) + (q-3)\sum_{j=1}^{k-1}\vk_4(j).$$
\end{proof}

Here, we collect the formulas for $A_{p,q}(k)$ in the hyperbolic case, $(p-2)(q-2)>4$.
\begin{equation}\label{eq:allAparams}
    \begin{split}            
        \nk(k)&= 1 + \frac{p(q-2)}{(\alpha-1)\sqrt{t^2-4}}  \left( \alpha^k + \frac{1}{\alpha^{k-1} } -\alpha-1   \right)  \\
        \vk(k)&=  \frac{p}{t-2} \left( \alpha^k+ \frac{1}{\alpha^k} -2\right)  \\
        \ek(k)    &=   \frac{p}{(\alpha-1)\sqrt{t^2-4}} \left( (\alpha+q-1)\alpha^k+\frac{\alpha q -\alpha+1}{\alpha^k} - q(\alpha+1) \right)    \\
        \ek_1(k) &=  \frac{p(\alpha+1)}{\sqrt{t^2-4}} \left( \alpha^{k-1} -\frac{1}{\alpha^k} \right)     \\
        \ek_2(k)&= \frac{p}{(\alpha-1) \sqrt{t^2-4}} \left( (\alpha q-\alpha+1) \alpha^{k-1} + \frac{q-1+\alpha}{\alpha^{k-1}} -q(\alpha+1) \right)     \\       
        \vk_{int}(k) & = \frac{p}{t-2} \left( \alpha^{k-1}+\frac{1}{\alpha^{k-1}} -2\right)
    \end{split}
\end{equation}

Now, we establish how to construct the sequences $d_{k,i}$. To begin with, we order the vertices on the boundary of $A_{p,q}(k)$, labeling them $x_{k,i}$ for $1\le i\le \ek_1(k)$. 

\begin{defn}  \label{defn:vertices}
For $A_{p,q}(1)$, label the vertices $x_{1,1},x_{1,2}, \dots , x_{1,p}$ in the counter-clockwise direction. For $k\geq 2$, suppose that we have labelled the perimeter vertices of $A_{p,q}(k-1)$, and let $T$ be the tile of $A_{p,q}(k)$ attached to both $x_{k-1,1}$ and $x_{k-1,\ek_1(k-1)}$, the first and last vertices of $A_{p,q}(k-1)$. Define $x_{k, 1}$ to be the vertex of $T$ adjacent to $x_{k-1,\ek_1(k-1)}$, and different from $x_{k-1,1}$. Then, the perimeter vertices of $A_{p,q}(k)$ are labelled in the counter-clockwise direction starting at $x_{k,1}$. See Figures \ref{fig:S37-4-5} and \ref{fig:A3-3-7}. 

\end{defn}

Now, recall from the proof of Proposition \ref{prop:recurrence} that each vertex in $A_{p,q}(k)$ generates vertices on the next layer according to its degree. Let $d_k$ denote the sequence $\{d_{k,i}\}$. We can compute the sequence $d_{k+1}$ from $d_k$ in the following manner. See the ensuing figures.

Firstly, $d_{1}$ is the sequence $2,2,\dots, 2$ containing $p$ elements. Then, $d_{k}$ is constructed from $d_{k-1}$ by replacing the $i$th element $d_{k-1,i}$ with the string of the degrees of the perimeter vertices of $A_{p,q}(k)$ associated to $x_{k-1,i}$. More precisely:
If $p\geq 4$ and $q\geq 4$, $d_{k-1,i}$ is replaced by
            $$\underbrace{  \underbrace{3,2,\dots,2}_{p-3}, \underbrace{3,2,\dots,2}_{p-2}, \dots , \underbrace{3,2,\dots,2}_{p-2}   
                }_{q-d_{k-1,i} \text{ groups of } 3,2,\dots,2}. $$
If $p=3$, and $q\geq 6$, $d_{k-1,i}$ is replaced by 
            $$\underbrace{4,3,\dots, 3}_{q-d_{k-1,i}-1}.$$
Finally, if $q=3$, and $p\geq 6$, $d_{k-1,i}=2$ is replaced by 
            $$\underbrace{3,2,\dots ,2}_{p-3} \text{ or } \underbrace{3,2,\dots ,2}_{p-4} $$
            depending whether $d_{k-1,i}=2$ is preceded by a 2 or a 3, respectively. On the other hand, $d_{k-1,i}=3$ is simply erased.

\vspace{.5cm}
\begin{ex} \label{ex:degree-sequence}
Here, we compute the sequences $d_k$ for the Euclidean cases and some hyperbolic ones. As expected, from Lemma \ref{lem:recurrence-vert}, we will encounter linear growth in the former and exponential in the latter. In \cite{beautifulanimals}, we also provide an applet to explore these examples and their numerical properties.
\vspace{.5cm}
    \begin{itemize}
        \item[A)] For $\{p,q\}=\{3,6\}$, the substitution rules are: $4\to 4$; $3 \to 4,3$; and $2\to 4,3,3$. The first sequences are
        \begin{align*}
        d_1 &= 2,2,2    &  \\
        d_2 &= 4,3,3, \dots  & (\text{3 times})    \\
        d_3 &= 4,4,3,4,3, \dots & (\text{3 times})     \\
        d_4 &= 4,4,4,3,4,4,3, \dots & (\text{3 times})
        \end{align*}
        It is easy to see that, for $k\geq 3$, $d_k$ consists of the sequence $\underbrace{4,4\dots, 4,3}_{k}\underbrace{4,4\dots, 4,3}_{k-1}$, repeated 3 times. 

  \begin{figure}[h]   
    \centering
   \includegraphics[scale=4.5]{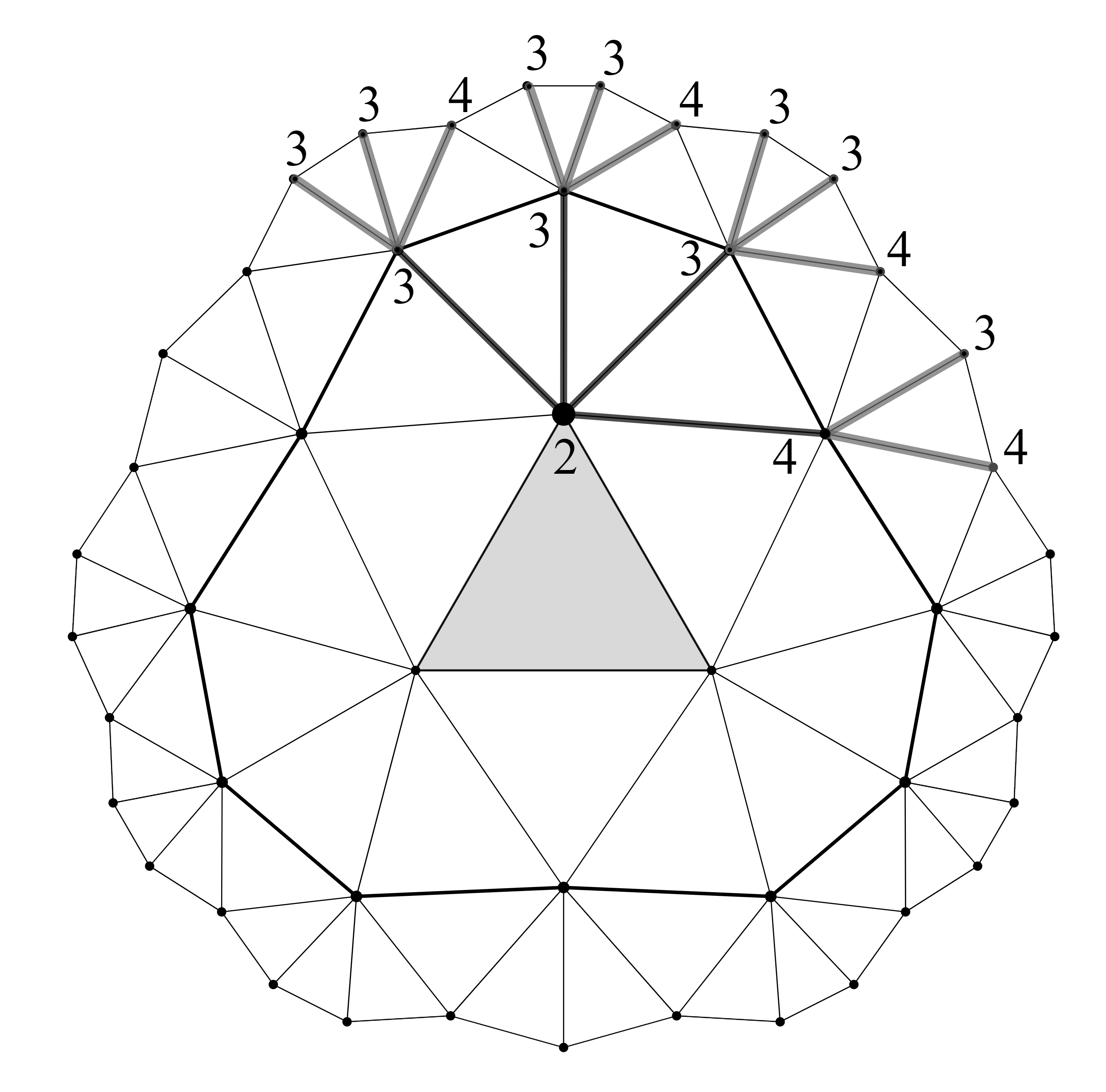}
    \caption{Graphic representation of the substitution rules for $\{p,q\}= \{3,7\}$. From $A_{3,7}(1)$ to $A_{3,7}(2)$ we apply $2\to 4,3,3,3$. From $A_{3,7}(2)$ to $A_{3,7}(3)$ we apply $3 \to 4,3,3$ and $4\to 4,3$. We observe that in this case we only have vertices of degree two in $A_{3,7}(1)$ but then all the vertices of $A_{3,7}(k)$ for $k>1$ will have degree $3$ or $4$.} \label{fig:A3-p3-substitution}
\end{figure}
\vspace{.5cm}
        \item[B)] For $\{p,q\}=\{3,7\}$, the substitution rules are: $4\to 4,3$; $3 \to 4,3,3$; and $2\to 4,3,3,3$. The first sequences are
        \begin{align*}
        d_1 &= 2,2,2     &  \\
        d_2 &= 4,3,3,3, \dots  & (\text{3 times})    \\
        d_3 &= 4,3,4,3,3,4,3,3,4,3,3, \dots & (\text{3 times})     \\
        d_4 &= 4,3,4,3,3,4,3,4,3,3,4,3,3,4,3,4,3,3,4,3,3,4,3,4,3,3,4,3,3, \dots & (\text{3 times})
        \end{align*}  
        
    \item[C)] For $\{p,q\}=\{4,4\}$, the substitution rules are: $3\to 3$, and $2 \to 3,3,2$. The first sequences are
        \begin{align*}
        d_1 &= 2,2,2,2     &  \\
        d_2 &= 3,3,2, \dots  & (\text{4 times})    \\
        d_3 &= 3,3,3,3,2, \dots & (\text{4 times})     \\
        d_4 &= 3,3,3,3,3,3,2, \dots & (\text{4 times})
        \end{align*}
        It is easy to see that, for $k\geq 2$, $d_k$ consists of the sequence $\underbrace{3,3\dots, 3,2}_{2k-1}$, repeated 4 times.
        
        \begin{figure}[h]   
    \centering
   \includegraphics[scale=4.5]{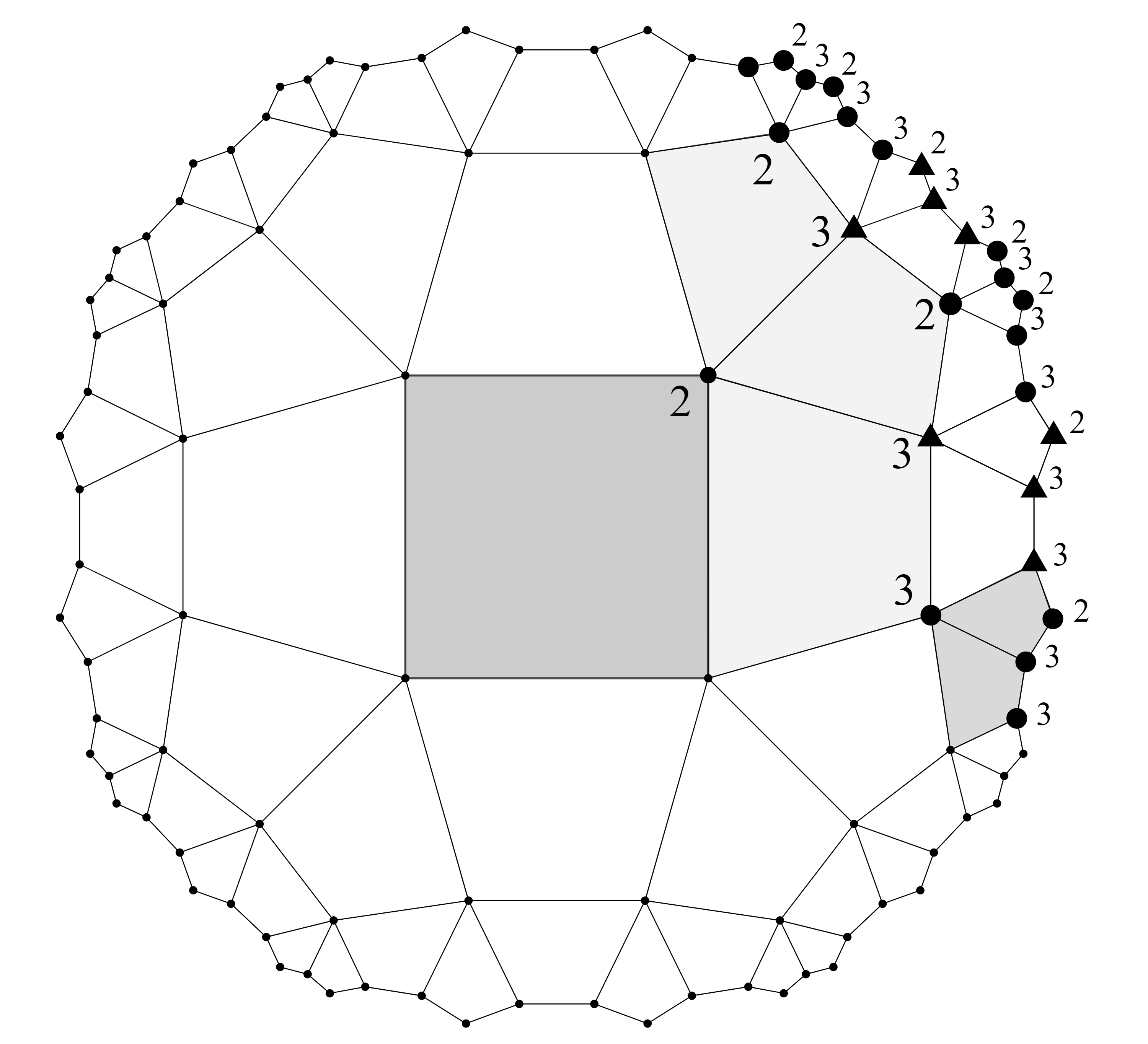}
    \caption{Graphic representation of the substitution rules for $\{p,q\}= \{4,5\}$. From $A_{4,5}(1)$ to $A_{4,5}(2)$ we apply $2\to 3,3,2,3,2$; the associated tiles are shaded in light gray. We alternate the symbols {\scriptsize $\blacktriangle$} and $\bullet$ to emphasize the substitution rules for each perimeter vertex of $A_{4,5}(2)$. } \label{fig:A3-4-5-substitution}
    \end{figure}
\vspace{.5cm}
        \item[D)] For $\{p,q\}=\{4,5\}$, the substitution rules are: $3 \to 3,3,2$, and $2\to 3,3,2,3,2$. The first sequences are
        \begin{align*}
        d_1 &= 2,2,2,2    &  \\
        d_2 &= 3,3,2,3,2, \dots  & (\text{4 times})    \\
        d_3 &= 3,3,2,3,3,2,3,3,2,3,2,3,3,2,3,3,2,3,2,  \dots & (\text{4 times})     
        \end{align*}      

 \vspace{.5cm}   
        \item[E)] For $\{p,q\}=\{6,3\}$, the substitution rules are: $3$ is erased; if a $2$ is preceded by a $3$, then $2 \to 3,2$; and if a $2$ is preceded by a $2$, then $2 \to 3,2,2$. The first sequences are
        \begin{align*}
        d_1 &= 2,2,2,2,2,2     &  \\
        d_2 &= 3,2,2, \dots  & (\text{6 times})    \\
        d_3 &= 3,2,2,3,2, \dots & (\text{6 times})     \\
        d_4 &= 3,2,3,2,2,3,2,2, \dots & (\text{6 times})
        \end{align*}
        It is easy to see that, for $k\geq 2$, $d_k$ consists of the sequence $3,2,\underbrace{3,2,2,3,2,2,\dots, 3,2,2}_{3,2,2 \text{ repeated }  k-2 \text{ times}} $, repeated 6 times.
        
           \begin{figure}[h]   
    \centering
   \includegraphics[scale=4]{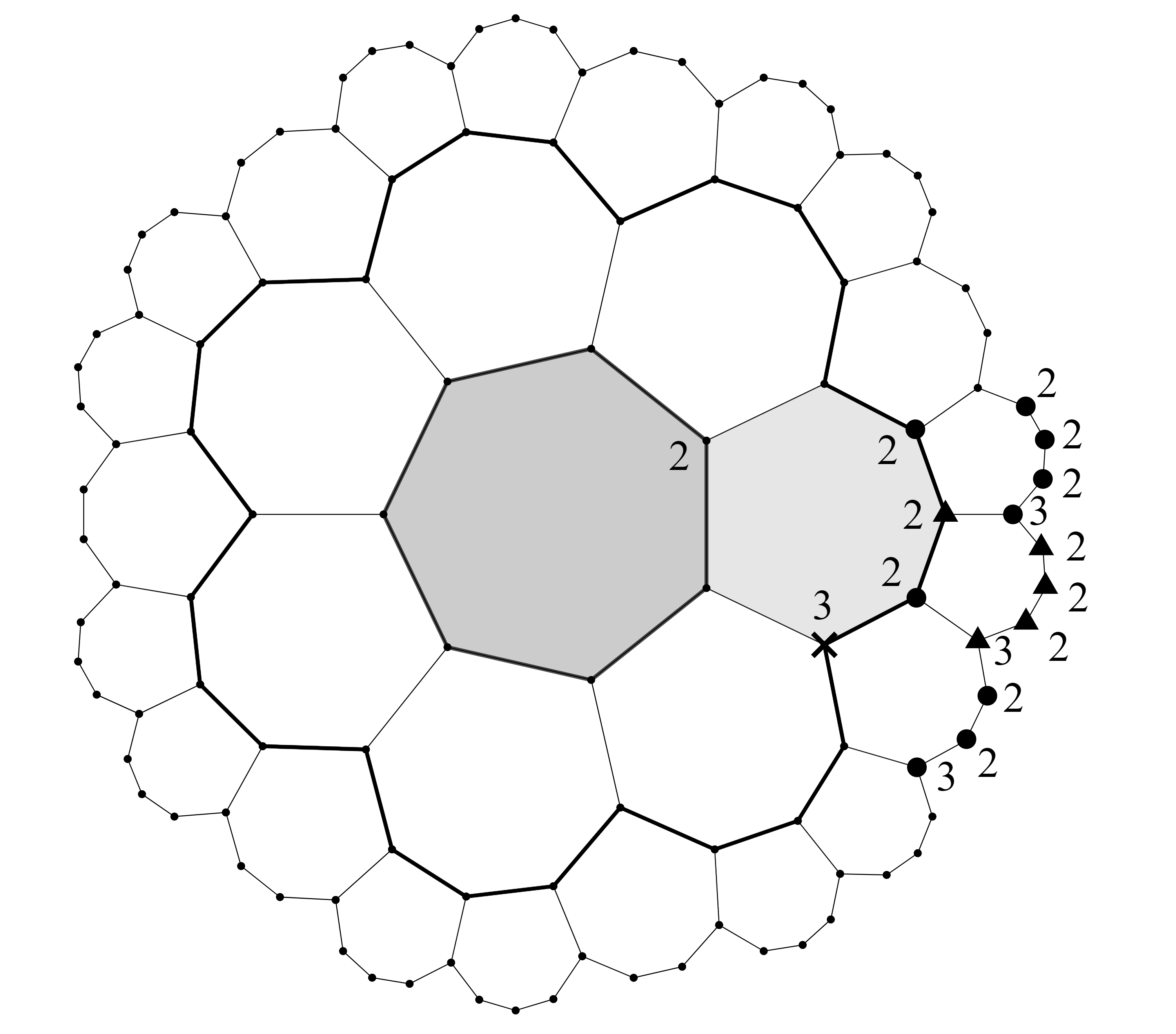}
    \caption{Graphic representation of the substitution rules for $\{p,q\}= \{7,3\}$. From $A_{7,3}(2)$ to $A_{7,3}(3)$, the 3 is erased (represented by $\times$), and  $2 \to 3,2,2,2$. We alternate the symbols {\scriptsize $\blacktriangle$} and $\bullet$ to emphasize the substitution rules for each vertex.}  \label{fig:A3-7-3-substitution} 
    \end{figure}
 \vspace{.5cm}   
    \item[F)] For $\{p,q\}=\{7,3\}$, the substitution rules are: $3$ is erased; if a $2$ is preceded by a $3$, then $2 \to 3,2,2$; and if a $2$ is preceded by a $2$, then $2 \to 3,2,2,2$. The first sequences are
        \begin{align*}
        d_1 &= 2,2,2,2,2,2,2     &  \\
        d_2 &= 3,2,2,2, \dots  & (\text{7 times})    \\
        d_3 &= 3,2,2,3,2,2,2,3,2,2,2  \dots & (\text{7 times})            \\
        d_4 &= 3,2,2,3,2,2,2,3,2,2,3,2,2,2,3,2,2,2,3,2,2,3,2,2,2,3,2,2,2, 
   \dots & (\text{7 times})     \end{align*}
 
    \end{itemize}
\end{ex}

\section{$\{p,q\}$-spiral animals}

In this section, we turn our attention to $\{p,q\}$-spiral animals, counting their graph parameters, and proving Theorem \ref{thm:e_1spiral}.

The $\{p,q\}$-spiral with $n$ tiles will be denoted by $S_{p,q}(n)$. Recall that we refer to its graph parameters as $\espir(n), \vspir(n), \espir_1(n)$, etc. The following procedure for constructing $S_{p,q}(n)$ gives the spiral in the counter-clockwise direction, as seen in Figure \ref{fig:S37-4-5}.

\begin{defn} \label{defn:spiral}
Let $n$ be a positive integer. Then $S_{p,q}(n)$ is constructed by adding tiles $T_1, T_2,\ldots, T_n$ one at a time such that $T_n$ is the unique tile added to $S_{p,q}(n-1)$ to obtain $S_{p,q}(n)$. If $n=\nk(k)$, then $S_{p,q}(n)=A_{p,q}(k)$. Otherwise, for $\nk(k) < n < \nk(k+1)$, first $T_{\nk(k)+1}$ is attached to $A_{p,q}(k)$ by gluing it along the edge $\{x_{k,1}, x_{k,\ek_1(k)}\}$. Then subsequent tiles $T_i$ are added by gluing along the perimeter edge of $T_{i-1}$ which contains $x_{k,1}$, until $x_{k,1}$ is saturated. Next, we move on to the vertex $x_{k,2}$, and we repeat this process until $n-\nk(k)$ tiles have been added to $A_{p,q}(k)$.
\end{defn}

For the rest of this section, we always assume that $\nk(k) < n \le \nk(k+1)$, and thus $T_n$ is a tile being added in the $(k+1)$-th layer. Similar to $A_{p,q}(k)$, we use the gluing parameter of each subsequent tile to chart the jumps in the sequences $\espir_1(n)$ and $\espir_2(n)$. Here, $T_n$ is added to $S_{p,q}(n-1)$, and so we say $T_n$ is $\epsilon$-glued if $\deg_{S_{p,q}(n-1)}(T_n)=\epsilon$, with $\epsilon\ge 1$ since $T_n$ always shares an edge with $T_{n-1}$. 

\begin{rmk}\label{rmk:e_2bounds}
If $T_n$ is $\epsilon$-glued to $S_{p,q}(n-1)$, then \begin{align*}
\espir_1(n)-\espir_1(n-1) &= p-2 \epsilon,\\ \espir_2(n)-\espir_2(n-1)&= \epsilon.
\end{align*}
\end{rmk}

This is a simple observation, as when $T_n$ is $\epsilon$-glued it adds $p-\epsilon$ perimeter edges while covering $\epsilon$ edges of $S_{p,q}(n-1)$, and each covered edge is added to the count for $e_2$.





With $T_n$ in the $(k+1)$-th layer, note that $\deg_{S_{p,q}(n-1)}(T_n)$ is also easily determined from $\deg_{A_{p,q}(k)}(T_n)$, differing by either 0, 1, or 2 additional edges. The first tile $T_{\nk(k)+1}$ has no additional shared edges with other tiles in the $(k+1)$-th layer, and is 1-glued for $q\geq 4$ and 2-glued for $q=3$ (the latter for $k\geq2$). Meanwhile, the last tile $T_{\nk(k+1)}$ has two additional shared edges; one with $T_{\nk(k+1)-1}$ and one with the first tile in this layer, $T_{\nk(k)+1}$. Thus, it is 2-glued for $q\geq 4$ or 3-glued for $q=3$. For any other value of $n$, $T_n$ only shares one additional edge with $T_{n-1}$, and so $\deg_{S_{p,q}(n-1)}(T_n)=\deg_{A_{p,q}(k)}(T_n)+1$. Combining this with the analysis in Proposition \ref{prop:recurrence} and Remarks \ref{rmk:vdeg} and \ref{rmk:e-gluing}, we can then determine the exact values of these jumps.

\begin{lem}
When $q > 3$ and for $n > 1$,
    \begin{align*}
        &\espir_1(n)-\espir_1(n-1) = p-2  \text{ or } p-4,\\
        &\espir_2(n)-\espir_2(n-1) = 1  \text{ or 2, respectively}.
    \end{align*}
When $q=3$ and for $n > 2$,
    \begin{align*}
        &\espir_1(n)-\espir_1(n-1) = p-4 \text{ or } p-6,\\
        &\espir_2(n)-\espir_2(n-1) = 2 \text{ or 3, respectively}.\\
    \end{align*}
\end{lem}

\begin{cor} \label{cor:e_2bounds}
For $n \ge 1$ and $l\ge 1$, we have the bounds 
    \begin{align*}
        &\espir_2(n+l)-\espir_2(n)\leq 2l \text{ when } q>3,\\
        &\espir_2(n+l)-\espir_2(n)\leq 3l \text{ when } q=3.\\
    \end{align*}
\end{cor}

\begin{proof}[Proof of Theorem \ref{thm:e_1spiral}]
Recall that $\espir_1(n)$ is the number of perimeter edges of $S_{p,q}(n)$, and $m$ is the number of perimeter vertices of $A_{p,q}(k)$ which are saturated in $S_{p,q}(n)$. 

When $T_n$ is $\epsilon$-glued, it saturates $\epsilon-1$ perimeter vertices of $A_{p,q}(k)$. So an $\epsilon$-glued tile increments $m$ by $\epsilon-1$, and $e_1$ by $p-2\epsilon=(p-2)-2$(increment of $m$). Then, for any $q$, we get
    $$\espir_1(n)= (p-2)(n-\nk(k)) + \ek_1(k) -2m,$$

which gives the expression in Theorem \ref{thm:e_1spiral}. Similarly, we have
    \begin{align*}
    \espir_2(n)&= \ek_2(k) + n-\nk(k)+m,     \\
    \espir(n) &= (p-1)(n-\nk(k))+\ek(k)-m,    \\
    \vspir(n)  &=(p-2)(n-\nk(k))+\vk(k)-m.    
    \end{align*}    
    
Now, we compute $m$ with the help of the sequence $d_{k,i}$. The vertex $x_{k,1}$ has degree $d_{k,1}$ in $A_{p,q}(k)$, hence it needs $q-d_{k,i}+1$ tiles to become saturated. So, if $n-\nk(k) \leq q-d_{k,1}$ then $m=0$. Now, for $i>1$, the vertex $x_{k,i}$ has degree $d_{k,i}$ in $A_{p,q}(k)$, but in the spiral there is already a tile attached to both $x_{k,i-1}$ and $x_{k,i}$; so $x_{k,i}$ needs $q-d_{k,i}$ tiles to become saturated. This analysis, together with the fact that $n-\nk(k)$ must be between the number of tiles needed to saturate $m$ vertices and the number of tiles needed to saturate $m+1$ vertices, gives us the inequality,
$$ 1+ \sum_{i=1}^m (q-d_{k,i}) \leq n-\nk(k) < 1+ \sum_{i=1}^{m+1} (q-d_{k,i}).$$

\end{proof}

We now turn our attention to the properties of the dual graph of $S_{p,q}(n)$. Recall that $n'$ is the number of faces in the dual graph, and we use $T_{n'}$ and $T_{n'+1}$ to denote the $n'$-th and $(n'+1)$-th tiles of the dual $\{q,p\}$-spiral. Additionally, recall that $e_{2,i}$ is the number of edges that are incident to $i$ faces in the dual graph, with $e_{2,0}$ tracking how far off $S'_{p,q}(n)$ is from being exactly a $\{q,p\}$-spiral.

\begin{lem} \label{lem:e_2,0bound}
The dual graph of $S_{p,q}(n)$ consists of the $\{q,p\}$-spiral animal $S_{q,p}(n')$ with a possibly trivial path of $e_{2,0}$-edges attached. Moreover, if $T_{n'+1}$ is $\epsilon$-glued to $S_{q,p}(n')$, then $\espir_{2,0}(n)\leq q-2-\epsilon$.
\end{lem}

\begin{proof}
Recall that dual faces, i.e., tiles of the dual animal, are in 1-to-1 correspondence with interior vertices of the original. Then by Definition \ref{defn:spiral}, we want to show that if $x_{k,i}$ is an interior vertex that corresponds to tile $T_{n'}$ in the dual, then vertex $x_{k,i+1}$ (or $x_{k+1,1}$ when $i=\ek_1(k)$) corresponds to the tile $T_{n'+1}$ in the dual. To briefly avoid confusion caused by the index, we write $x_{*}$ to denote whichever of $x_{k,i+1}$ or $x_{k+1,1}$ is the next vertex.

For $n \le q$, this is trivially true, and covering the vertex $x_{1,1}$ of $T_1$ gives $S'_{p,q}(q)=S_{q,p}(1)$. For larger $n$, this follows from Definition \ref{defn:spiral}, since tiles are added in the same direction as the increments in the vertex labels. For instance, it is clear that the tiles represented by $x_{k,i}$ and $x_{*}$ will be adjacent, and the tile $T$ of $A_{p,q}(k)$ containing the edge $\{x_{k,i},x_{*}\}$ corresponds to the next open perimeter vertex in the dual, to which $T_{n'+1}$ is being glued. Furthermore, when $q > 3$, 1-glued tiles correspond to adding a path of $e_{2,0}$-edges in the dual. This path gets closed at the next 2-glued tile, which saturates $x_{*}$ and completes $T_{n'+1}$.

The bound follows from counting the extant edges of $T_{n'+1}$ in the dual. There are $\epsilon$ edges already present in $S_{q,p}(n')$, where it is going to be attached. The vertex representing the 2-glued tile in the original $\{p,q\}$-spiral completes the tile $T_{n'+1}$, adding two edges simultaneously. So these $\epsilon+2$ edges cannot be in the path tracing out the $q$-gon $T_{n'+1}$. 
\end{proof}

\section{Proof of Theorem \ref{thm:extremalspiral}}
We now work towards proving Theorem \ref{thm:extremalspiral} and the extremality of the graph properties of $S_{p,q}(n)$. Recall that an animal $A$ with $n$ tiles is said to be \textit{extremal} if $e_2(A)$ and $v_{int}(A)$ are maximal, and $e(A), e_1(A) $ or $v(A)$ are all minimal in the set of animals with $n$ tiles.


First, we show the maximality of $\espir_2(n)$ within the class of $\{p,q\}$-graphs with $n$ faces. This requires the following lemma, which asserts that if this holds for the dual structure, then it can be lifted to the original graph.

\begin{lem}\label{lem:e2_dual_lift}
Let $G$ be a $\{p,q\}$-graph with $n$ faces and $v_{int}$ interior vertices. If $e_{2}\left(G'\right)\leq e_2\left(S_{q,p}\left(v_{int}\right)\right)$, then $e_2(G)\le e_2(n)$.
\end{lem}

\begin{proof}
Recall that $e_{2,2}(G)=e_2\left(G'\right)$, which by assumption is at most $e_2\left(S_{q,p}\left(v_{int}\right)\right)$. Then, we can plug this into Equation (\ref{eqn:interior-edges}), 
    \begin{align*}
        e_2(G) &= \frac{q\left(n-c'\right) -e_{2,0}(G) +e_{2,2}(G)}{q-1},   \\
                & \leq \frac{q\left(n-c'\right) + e_{2,2}(G)}{q-1}\\
                & \leq \frac{q\left(n-c'\right) + e_{2}\left(S_{q,p}\left(v_{int}\right)\right)}{q-1}.
    \end{align*}
    
To avoid a cacophony of indices and parentheses, we abbreviate our notation when referring to the dual $\{q,p\}$-spiral and write $\edual_2(i)=e_2\left(S_{q,p}(i)\right)$. Then applying Euler's identity for $G'$ to replace $v_{int}$, we have
    $$    e_2(G)  \leq \frac{q\left(n-c'\right) + \edual_2\left(v_{int}\right)}{q-1} = \frac{q\left(n-c'\right) + \edual_2\left(e_2(G)-n+c'\right)}{q-1}.$$
We can reduce this to the case $c'=1$ by using Corollary \ref{cor:e_2bounds} to obtain 
$$\edual_2\left(e_2(G)-n+c' \right)=\edual_2\left(e_2(G)-n+1+c'-1\right) \leq \edual_2\left(e_2(G)-n+1\right)+3\left(c'-1\right).$$
And similarly adding and subtracting 1 to $n-c'$ in the first term, we get 
    \begin{equation}\label{eq:defn_gn}
    \begin{split}
    e_2(G) & \leq  \frac{q\left(n-c'\right) + \edual_2\left(e_2(G)-n+c'\right)}{q-1}, \\
            & \leq  \frac{q(n-1) - q\left(c'-1\right) + \edual_2\left(e_2(G)-n+1\right)+ 3\left(c'-1\right)}{q-1}, \\
            &\leq \frac{q(n-1) + \edual_2\left(e_2(G)-n+1\right)}{q-1}. 
    \end{split}        
    \end{equation}
The last inequality follows from the fact that $q\ge3$, so the coefficient of $\left(c'-1\right)$ is at most 0.

Now we define the sequence $$g_n(j)=\edual_2(j-n+1) -(q-1)j +q(n-1),$$ noting that a simple rearrangement of the inequality in Equation (\ref{eq:defn_gn}) yields $g_n(e_2(G))\geq 0$. To prove $e_2(G)\leq \espir_2(n)$, we will show that $g_n$ is non-increasing, and $j=\espir_2(n)$ is the last value where it is non-negative.

That $g_n$ is non-increasing is primarily a consequence of Corollary \ref{cor:e_2bounds} applied to the $\{q,p\}$-spiral $S_{q,p}(j-n+1)$. Using $\noep$ as the multiplicative factor, which here depends on $p$ instead of $q$,
\begin{align*}
        g_n(j+1) &= \edual_{2}(j-n+1+1) -(q-1)(j+1) +q(n-1),     \\
                & \leq \edual_{2}(j-n+1)+\noep - (q-1)j -(q-1) + q(n-1), \\
                & = g_n(j)-q+\noep+1. 
\end{align*}
Recall that $\noep= 2$ if $p>3$ and $\noep=3$ if $p=3$. In the first case, $q\geq 3$, and in the second $q\geq 6$, so we always get $g_n(j+1)\leq g_n(j)$.

Then $g_n(e_2(n))=g_n\left(e_2\left(S_{p,q}(n)\right)\right)\ge0$ follows from $g_n\left(e_2(G)\right)\geq 0$ for any $\{p,q\}$-graph $G$ with $n$ faces, and it only remains to show that $g_n\left(\espir_2(n)+1\right) < 0$. Recall that the dual of $S_{p,q}(n)$ is connected by definition, and is a $\{q,p\}$-spiral with $n'$ faces and possibly some additional edges. Hence $c'=1$, and by Euler's formula the number of faces in the dual is $n'=1-n+\espir_2(n)$, with $\edual\left(n'\right)=\espir_{2}(n)$. Now, we can apply Remark \ref{rmk:e_2bounds} to $\edual_{2}\left(1-n+\espir_2(n)+1\right) = \edual_{2}\left(n'+1\right)$. For $\epsilon$ the gluing parameter of the tile $T_{n'+1}$ in the $\{q,p\}$-spiral, we get that

\begin{align*}
            g_n\left(\espir_2(n)+1\right) &= \edual_{2}\left(n'+1\right) -(q-1)\left(\espir_2(n)+1\right) +q (n-1),   \\
                    &= \edual_{2}\left(n' \right) +\epsilon -(q-1)\espir_2(n) -(q-1) +q (n-1).  
\end{align*}
And Equation (\ref{eqn:interior-edges}) for $G=S_{p,q}(n)$ yields that
    \begin{align*}
    \espir_{2,0}(n) &= \espir_{2,2}(n)-(q-1)\espir_2(n)+q(n-1),\\
    &= \edual_{2}\left(n' \right) -(q-1)\espir_2(n) +q (n-1).    
    \end{align*}
So we have
    \begin{align*}
            g_n\left(\espir_2(n)+1\right) =\espir_{2,0}(n)+\epsilon - (q-1).
    \end{align*}    
Finally, Lemma \ref{lem:e_2,0bound} tell us that $\espir_{2,0}(n)\leq q-2-\epsilon$, and hence $g_n\left(\espir_2(n)+1\right)\leq -1$. Therefore, since $g_n\left(e_2(G)\right)\ge 0$, it must be that $e_2(G)\le\espir_2(n)$.
\end{proof}

We use induction and the bounds from Lemma \ref{lem:v_intbound} to show that this lifting can always be done, and $e_2(n)$ is in fact always maximal in the class of $\{p,q\}$-graphs.

\begin{lem}\label{lem:e2-pq-no-holes}
Fix a pair $\{p,q\}$ and $n\ge 0$. If $G$ is a $\{p,q\}$-graph with $n$ faces, then $e_2(G)\leq e_2(n)$.
\end{lem}

\begin{proof}
We proceed by induction, leveraging the dual structure and Lemma \ref{lem:e_2,0bound}. Let $\I(p,q,n)$ be the statement: $$\text{``If $G$ is a $\{p,q\}$-graph with exactly $n$ faces, then $e_2(G) \le e_2(n)$.''}$$ Our initial base case will be that $\I(p,q,n)$ holds for all $p,q$ and $n$ such that $n\leq q$. When $n < q$, there cannot be any interior vertices. So the dual graph of $G$ is either a tree, a disjoint union of trees, or the empty graph on $n$ vertices, and $e_2(G)=e(G') \leq n-1 =\espir_2(n)$. Similarly, when $n=q$, the optimal value is obtained when the dual graph is a $q$-cycle, then $e_2(G)\leq q =\espir_2(q)$.

Next, we observe that Lemma \ref{lem:e2_dual_lift} asserts that $\I(p,q,n)$ follows from $\I(q,p,v_{int})$, where we slightly abuse notation and allow $v_{int}$ to represent any viable value for $v_{int}$ given a $\{p,q\}$-graph with $n$ faces. To make this more rigorous, we combine this with the bounds in Lemma \ref{lem:v_intbound} for $v_{int}$ in terms of $n$ to get the following three facts:\\
\begin{itemize}
\item[(1)] For $p,q\ge 4$, $\I(p,q,n)$ follows from the set $\{\I(q,p,i): i\le n-1\}$.
\item[(2)] For $p=3$, $\I(3,q,n)$ follows from the set $\{\I(q,3,i): i\le n/2\}$.
\item[(3)] For $q=3$, $\I(p,3,n)$ follows from the set $\{\I(3,p,i): i\le 2n-2\}$.\\
\end{itemize}

First, we show $\mathcal{I}(p,q,n)$ holds for all $n$ in the case that both $p,q\ge4$. Fix $q\ge 4$. Then, for any $p > q$ and any $q < n \le p$, we have that $v_{int} < p$ by Lemma \ref{lem:v_intbound}. Hence $\I(q,p,v_{int})$ is satisfied by the base cases, and for any pair $p,q\ge 4$, we conclude $\I(p,q,n)$ holds for all $n\le\max(p,q)$.

Now suppose that for all pairs $p,q\ge4$, the statement $\I(p,q,i)$ holds for all $i\le \max(p,q)+j$ for some fixed integer $j\ge 0$. Then, for any fixed $p$ and $q$, let $n=\max(p,q)+j+1$. Thus $\I(p,q,n)$ follows by statement (1) and the induction assumption, and applying this to each ordered pair $(p,q)$ allows us to increment $j$ by 1 in the assumption. Therefore $\I(p,q,n)$ holds for all $n$ and all $p,q\ge 4$.

Next, fix one parameter at 3, and denote the other by $r \ge 6$. Then $\I(r,3,n)$ follows from $\I(3,r,v_{int})$ if we set $2n-2 \le r$, and hence $n\le (r+2)/2$. Conversely, for $\I(3,r,n)$ we need up to $\I(r,3,n/2)$. By the previous statement for the $q=3$ case, we know $\I(r,3,n/2)$ holds if $n/2\le (r+2)/2 \Rightarrow n\le r+2$. Hence for every $r\ge 6$, we have that $\I(r,3,n)$ holds for $n\le(r+2)/2$, and  $\I(3,r,n)$ holds for $n\le r+2$. Using these as the base cases, we zipper them together to perform the induction.

Fix any $r\ge6$, and suppose that for a fixed $j\ge 0$, $\I(r,3,i)$ holds for all $i\le(r+2)/2+j$, and additionally $\I(3,r,i)$ holds for all $i\le r+2+2j$. Let $n=(r+2)/2+j+1$. Then a graph considered in $\I(r,3,n)$ has $v_{int}\le 2n-2 = r+2+2j$. Similarly, for $n=r+2+2j+1$, a graph considered in $\I(3,r,n)$ has $v_{int}\le n/2 = (r+2)/2+j+1/2$. For $n=r+2+2j+2$, the bound is $v_{int}\le n/2 = (r+2)/2+j+1$. Hence, adding 1 to $n$ in the $q=3$ case is covered by the assumption in the $p=3$ case, and adding either 1 or 2 in the $p=3$ case is covered by the assumption in the $q=3$ case. Thus, we can increment $j$ by 1 for every $r\ge 6$ in both cases, and by induction $\I(r,3,n)$ and $\I(3,r,n)$ hold for all $r\ge 6$ and all $n\ge0$. 

\end{proof}

The maximality of $\espir_2(n)$ can immediately be churned into the full litany of extremal properties of spiral animals in the class of $\{p,q\}$-graphs. As a corollary, we obtain a proof of the statement of Theorem \ref{thm:extremalspiral} in the particular case for animals with no holes.

\begin{cor}\label{cor:extspir-no-holes}
Let $A$ be a $\{p,q\}$-animal with $n$ tiles and no holes. Then $e_2(A)\leq \espir_2(n)$, $v_{int}(A)\leq \vspir_{int}(n)$, $e_1(A)\geq \espir_1(n)$, $e(A)\geq \espir(n)$, and $v(A)\geq \vspir(n)$.
\end{cor}
\begin{proof} 
Any $\{p,q\}$-animal $A$ with no holes can be regarded as a $\{p,q\}$-graph. Lemma \ref{lem:e2-pq-no-holes} implies that $e_2(A)\leq \espir_2(n)$, and the remaining inequalities are obtained from the relations between the graph parameters: $e=e_1+e_2=p\cdot n-e_2$; $e_1=p\cdot n - 2e_2$; $v_{int}=1+e_2-n$; and $v=1-n+e$.
\end{proof}

To prove Theorem \ref{thm:extremalspiral}, it remains to show that this corollary also holds for animals with holes. We will use induction on the number of holes to complete the proof, but first we require the following lemma stating that the number of interior edges in $S_{p,q}(n)$ plus the total number of edges in $S_{p,q}(l)$ is at least the number of interior edges in $S_{p,q}(n+l)$.  

\begin{lem}\label{lem:add-l}
For any fixed $n > 1$ and fixed $l\ge 1$, 
$$\espir_2(n)+\espir(l) \geq \espir_2(n+l).$$
\end{lem}
\begin{proof} Note that $\espir(l) = p\cdot l - \espir_2(l)$. By Corollary \ref{cor:e_2bounds} along with the base case $\espir_2(1)=0$,  we know that $\espir_2(l) \leq \noep \cdot l$, where $\noep=2$ for $q>3$, and $\noep=3$ for $q=3$. 
Now consider the case $p\geq 4$, in which $p\geq 2\noep$ (if $\noep=2$ then $q>3$ and $p\geq 4$, if $\noep=3$ then $q=3$ and $p\geq6$). Then we have that 
    \begin{align*}
        \espir_2(n)+\espir(l) &= \espir_2(n)+p\cdot l - \espir_2(l), \\
                                    &\geq \espir_2(n)+p\cdot l - \noep \cdot l, \\
                                    &\geq \espir_2(n)+\noep \cdot l, \\ &\geq \espir_2(n+l).
    \end{align*}


The last inequality is again Corollary \ref{cor:e_2bounds}, this time applied to $S_{p,q}(n)$.

When $p=3$, we need a sharper analysis. Note that $p=3$ forces $q\ge 6$, and so all tiles are either 1- or 2-glued. Therefore, the increments of the sequence $\espir_2(l)-l$ are either 0 or 1, respectively. Furthermore, up to a small offset due to the initial value being $\espir_2(1)-1 = -1$, this sequence counts the number tiles that are 2-glued. Recall that these 2-glued tiles are exactly those which completely surround a perimeter vertex, forcing it to the interior. Since for $p=3$ the perimeter vertices of $A_{p,q}(k)$ have degree $\leq 4$, we need at least $q-3$ tiles to surround a perimeter vertex, the first and last of which are always 2-glued. Thus the increments of 1 are separated by at least $q-5$ increments of 0. 
On the other hand, for fixed $n$ the sequence $\espir_2(n)+2l-\espir_2(n+l)$ counts the number of 1-glued tiles added when constructing $S_{p,q}(n+l)$ starting from $S_{p,q}(n)$. The increments of this sequence are again 0 or 1, but now the increments by 0 are the 2-glued tiles. Since 1-glued tiles are added at a consistent ratio of at least $q-5$ to 1, and $\espir_2(l)-l$ actually starts at -1, we conclude that
\begin{align*}
    \espir_2(n)+2l-\espir_2(n+l) &\geq \espir_2(l)-l,\\
    &= 3l - \espir(l) - l.
\end{align*}
The equality here follows from the edge partition formula $\espir(l)=3l-\espir_2(l)$, and then a simple rearrangement of the inequality gives the desired result.



\end{proof}

\begin{proof}[Proof of Theorem \ref{thm:extremalspiral}]
Suppose that $e_2(n)$ is maximal in the class of $\{p,q\}$-animals with $h$ holes, and let $A$ be a $\{p,q\}$-animal with $n$ tiles and $h+1$ holes. Let $H$ be one of the holes of $A$, where $H$ itself is regarded as a $\{p,q\}$-animal, and let $l$ be the number of tiles of $H$. Denote by $\bar{A}$ the $\{p,q\}$-animal obtained by filling in $H$ in $A$. This is an animal with $n+l$ tiles and $h$ holes, so by the induction hypothesis $\espir_2(n+l)  \geq e_2\left(\bar{A}\right)$. Observe that $e_2\left(\bar{A}\right)=e_2(A)+e(H)$, and thus Lemma \ref{lem:add-l} implies that
    \begin{align*}
        \espir_2(n)+\espir(l) &\geq \espir_2(n+l),\\ 
                        & \geq e_2\left(\bar{A}\right),\\
                        & = e_2(A)+e(H),\\
                        & \geq e_2(A)+\espir(l).
    \end{align*}
Note that $e(H)\geq \espir(l)$, since $H$ is a $\{p,q\}$-animal with no holes. Therefore, by induction we conclude that $\espir_2(n)\geq e_2(A)$ for any $\{p,q\}$-animal $A$ with $n$ tiles. Then the other inequalities follow, as they did for Corollary \ref{cor:extspir-no-holes}, due to the equalities $e=p\cdot n-e_2$, $e_1=p\cdot n-2e_2$, $v_{int}=1+e_2-n$, and $v=1-n+e$.
\end{proof}

Given the extremality of all the pertinent graph parameters, we now briefly show how to derive the equations from \cite{harary1976extremal} for the Euclidean cases.

\begin{proof}[Proof of Corollary \ref{cor:Harary}] 
For the case $\{p,q\}=\{3,6\}$ we know $\nk(k)=6k^2-6k+1$ and, from Example \ref{ex:degree-sequence}, the sequence $d_k$ is equal to $\underbrace{4,4\dots, 4,3}_{k}\underbrace{4,4\dots, 4,3}_{k-1}$, repeated 3 times, for $k\geq 3$. From these values we deduce the following formula for $\espir_1(n)$ valid for $k\geq 3$.
    \begin{align*}
        \espir_1(n) = 
        \begin{cases}
    6k-3 + \bmod(n-1,2) & \text{if $0\leq n-\nk(k)\leq 1 $} \\
    6k-2 + \bmod(n,2)  & \text{if $2\leq n-\nk(k)\leq 2k-1$}  \\
    6k-1 + \bmod(n-1,2)  & \text{if $2k\leq n-\nk(k)\leq 4k-1$}  \\
    6k + \bmod(n,2)  & \text{if $4k\leq n-\nk(k)\leq 6k$}  \\    
    6k+1 + \bmod(n-1,2)  & \text{if $6k+1\leq n-\nk(k)\leq 8k-1$}  \\    
    6k+2 + \bmod(n,2)  & \text{if $8k\leq n-\nk(k)\leq 10k$}  \\    
    6k+3 + \bmod(n-1,2)  & \text{if $10k+1\leq n-\nk(k)\leq 12k-1$}  \\    
  \end{cases}
  \end{align*}

Then it is a straightforward computation to check that it can be written as
    \begin{align*}
        \espir_1(n) &= 2 \left\lceil \frac12( n+\sqrt{6n}) \right\rceil - n      
    \end{align*}
A similar analysis holds for the cases $\{p,q\}=\{4,4\}$ and $\{6,3\}$.
\end{proof}

\section{Enumeration of extremal $\{p,q\}$-animals}

In 2008 \cite{kurz2008counting}, Kurz enumerated polyominoes, i.e. $\{4,4\}$-animals, that attain minimum perimeter for any fixed number of tiles. Moreover, he proved that for all $l\geq 1$ there exists a unique polyomino with $l^2$ tiles which is extremal, and also a unique polyomino with $l(l+1)$ tiles. The former are the square shape polyominoes with side length $l$, and the latter are pronic rectangles with side lengths $l$ and $l+1$. In the general $\{p,q\}$ setting, we find that the sequence $A_{p,q}(k)$ corresponds precisely to squares with odd side lengths, and we define three more sequences corresponding to even squares and the two types of pronic rectangles distinguished by the parity of the shorter side. We are then able to generalize Kurz's result and show that each of these sequences yields unique extremal $\{p,q\}$-animals for any $p$ and $q$.

First we define the sequence $B_{p,q}(k)$ analogously to $A_{p,q}(k)$; but start with $B_{p,q}(1)=x_0$, a single vertex. Then $B_{p,q}(2)$ is the set of $q$ tiles incident to $x_0$, and $B_{p,q}(3)$ is the set of all tiles incident to the perimeter of $B_{p,q}(2)$, and so on. These form another distinct subsequence of spirals that are related to the $A_{p,q}(k)$ through the dual graph. Observe that the $B_{p,q}(k)$ sequence corresponds directly to the square shaped polyominoes with even side length in the $\{4,4\}$ case.



\begin{rmk}
The dual of $A_{p,q}(k)$ is $B_{q,p}(k)$, and thus $B_{p,q}(k)=S_{p,q}(n)$ for $n=v_{int}\left(A_{q,p}(k)\right)$. Similarly, the dual of $B_{p,q}(k)$ is $A_{q,p}(k-1)$.
\end{rmk}

We will use the duality of these two structures to build an inductive proof of their uniqueness, but first, we require a short proposition establishing a condition under which, up to isometries, an animal and its dual graph are in 1-to-1 correspondence.

\begin{prop}\label{prop:nice_dual}
Let $A$ be a $\{p,q\}$-animal such that $e_{2,0}(A)=0$. If $M$ is a $\{p,q\}$-animal with $M'=A'$, then $M=A$.
\end{prop}

\begin{proof}
The faces of $A'$ always represent a fixed set of vertices in the $\{p,q\}$-lattice which are saturated in $A$. By definition, a vertex is saturated if and only if all $q$ incident tiles are present in $A$. But if $e_{2,0}(A)=0$, then every vertex of $A'$ is incident to a face of $A'$, and hence every tile of $A$ is incident to a saturated vertex of $A$. So a tile is in $A$ if and only if it is incident to a saturated vertex. But then $M'=A'$ implies that also $e_{2,0}(M)=0$, so by the same reasoning a tile is in $M$ if and only if it is incident to a saturated vertex. Therefore, since $A$ and $M$ have the same set of saturated vertices, it must be that $M=A$.
\end{proof}

\begin{thm}\label{thm:uniq_layered}
For any $k\ge1$, $A_{p,q}(k)$ is the unique extremal $\{p,q\}$-animal with $\nk(k)$ tiles. Similarly, the same is true for $B_{p,q}(k)$ for $k\ge 2$, where the number of tiles is computed from $S_{p,q}(n)$ for $n=v_{int}\left(A_{q,p}(k)\right)$.
\end{thm}

\begin{proof}[Proof of Theorem \ref{thm:uniq_layered}]
The statement is trivially true for $A_{p,q}(1)$ for all $p$ and $q$. The induction will proceed by zippering the two sequences, with $B_{p,q}(k)$ following from $A_{q,p}(k-1)$ and then $A_{p,q}(k)$ following from $B_{q,p}(k)$.

Suppose that $A_{p,q}(k-1)$ is a unique extremal animal for all $p$ and $q$. Then for any fixed $p$ and $q$, let $B$ be a $\{p,q\}$-animal with $n(B)=n\left(B_{p,q}(k)\right)$ and perimeter $e_1(B)=e_1\left(B_{p,q}(k)\right)$. By the relations in the proof of Corollary \ref{cor:extspir-no-holes}, we then have that $B$ has the same number of interior edges and interior vertices as $B_{p,q}(k)$. Thus $B'$ has the same number of edges and faces as the dual of $B_{p,q}(k)$, which is $A_{q,p}(k-1)$.


Since $A_{q,p}(k-1)$ is itself a spiral, it's number of edges $e_{2,2}\left(B_{p,q}(k)\right)$ is maximal for its number of faces. Hence $e_{2,2}(B)\le e_{2,2}\left(B_{p,q}(k)\right)$. Then by Equation (\ref{eqn:interior-edges}) with $c'=1$ in both cases, it must be that $e_{2,0}(B)=0$, and thus $e_{2,2}(A) = e_{2,2}\left(B_{p,q}(k)\right)$. So the dual of $B$ has the same number of faces, edges, and interior edges as $A_{q,p}(k-1)$. And by our inductive assumption $A_{q,p}(k-1)$ is the unique animal with these parameters, and it must be that $B'=A_{q,p}(k-1)$. Hence, since $e_{2,0}=0$, by Proposition \ref{prop:nice_dual} we can lift this equality to get $B=B_{p,q}(k)$, which is therefore the unique extremal animal with $n=v_{int}\left(A_{q,p}(k)\right)$ tiles.

The exact same argument suffices to show that the uniqueness of $A_{p,q}(k)$ follows from knowing uniqueness for $B_{q,p}(k)$. Therefore, by induction we have uniqueness of $A_{p,q}(k)$ for all $k\ge1$, and for $B_{p,q}(k)$ for all $k\ge 2$.
\end{proof}


To extend now the pronic case, we define two more sequences in a similar fashion to the $A_{p,q}(k)$ and $B_{p,q}(k)$. Let $C_{p,q}(1)$ be the $\{p,q\}$-animal which has two tiles glued along a single adjoining edge, and let $D_{p,q}(1)$ be the $\{p,q\}$-graph which is simply a single edge connecting two vertices. Then the sequences $C_{p,q}(k)$ and $D_{p,q}(k)$ are defined analogously, with all possible tiles incident to the perimeter added in each subsequent step.

    \begin{figure}[h]  
    \centering
   \includegraphics[scale=2.2]{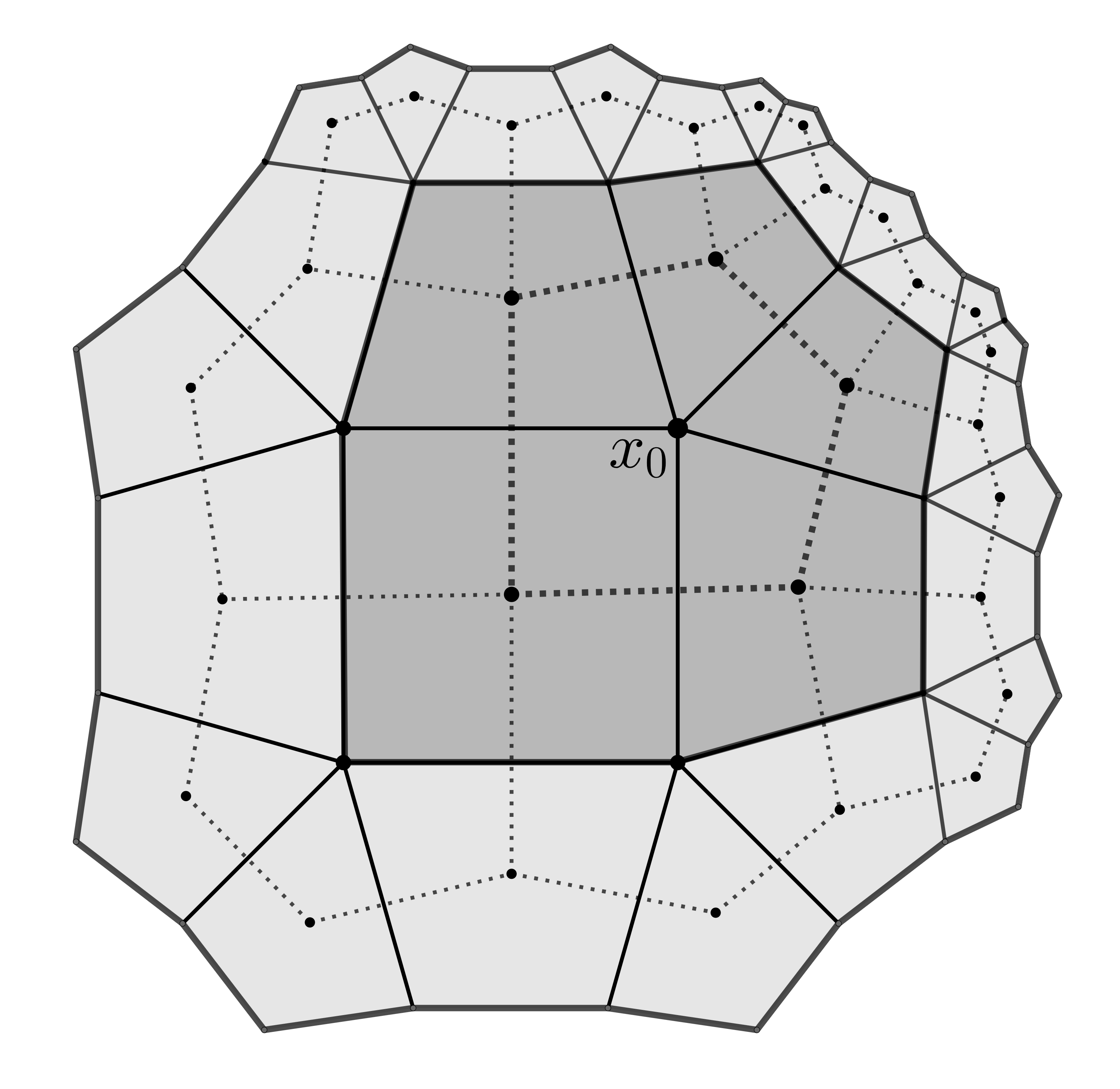} \hspace{0.5cm}
   \includegraphics[scale=2.5]{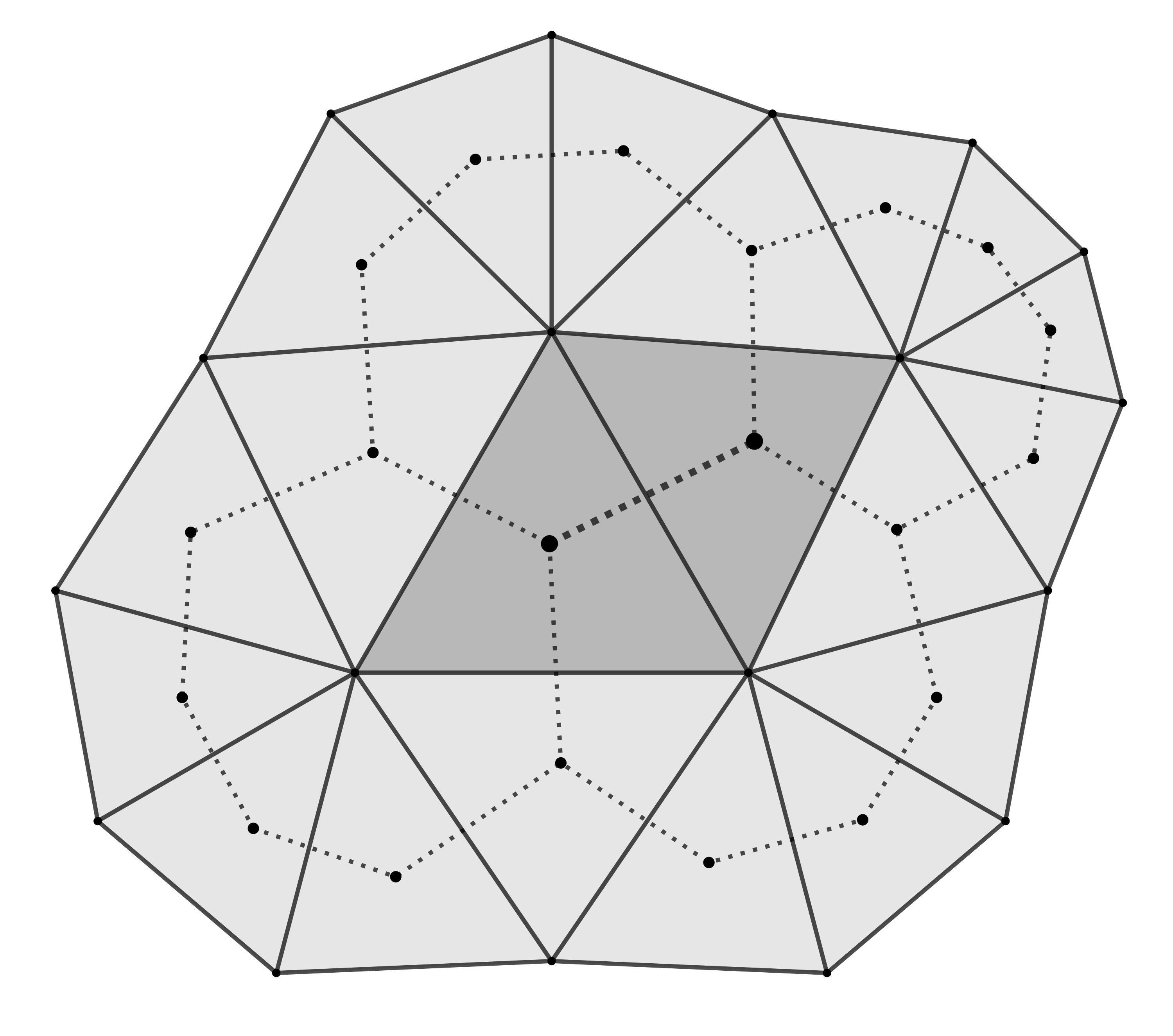} \hspace{0.5cm}
   \includegraphics[scale=2.5]{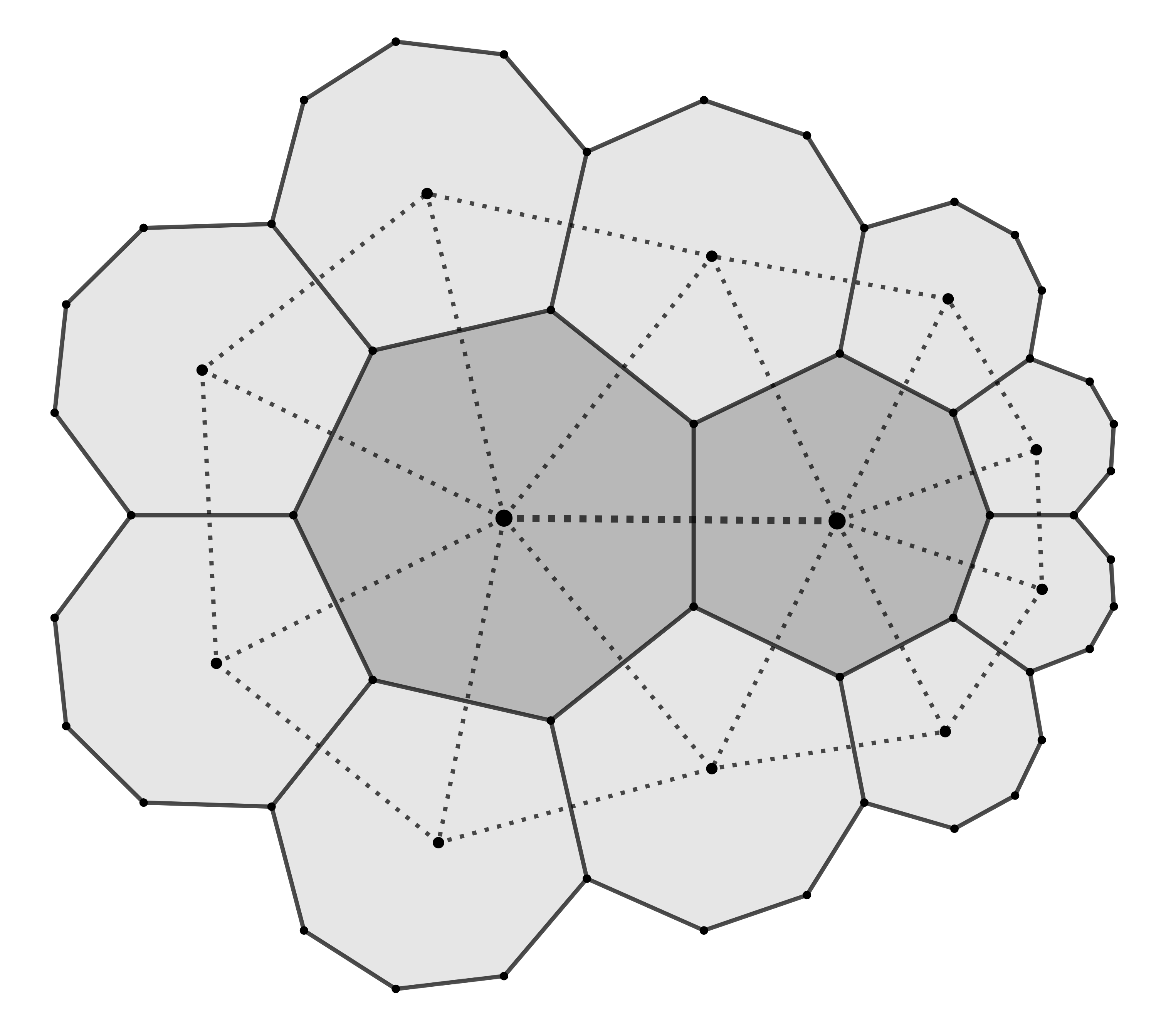}
    \caption{\textit{Left:} $B_{4,5}(3)$ and its dual $A_{5,4}(2)$. \textit{Center:} $C_{3,7}(2)$ and its dual $D_{7,3}(2)$. \textit{Right:} $C_{7,3}(2)$ and its dual $C_{3,7}(2)$ }  \label{fig:CD}
    \end{figure}
    
\begin{thm}\label{thm:pronic}
For any $k\ge1$, $C_{p,q}(k)$ is a unique extremal animal. Similarly, for any $k\ge2$, $D_{p,q}(k)$ is a unique extremal animal.
\end{thm}

\begin{proof}
 Note that as in the previous case, both of these sequences are subsequences of the spiral $S_{p,q}(n)$, and thus are extremal. Furthermore, the dual of $C_{p,q}(k)$ is $D_{q,p}(k)$, and the dual of $D_{p,q}(k)$ is $C_{q,p}(k-1)$. It is clear that up to isometries that there is only one $\{p,q\}$-animal on two tiles, and hence $C_{p,q}(1)$ is unique. The proof then follows by the same inductive argument as in the proof of Theorem \ref{thm:uniq_layered}.
\end{proof}

    \begin{figure}[h]  
    \centering
   \includegraphics[scale=1.5]{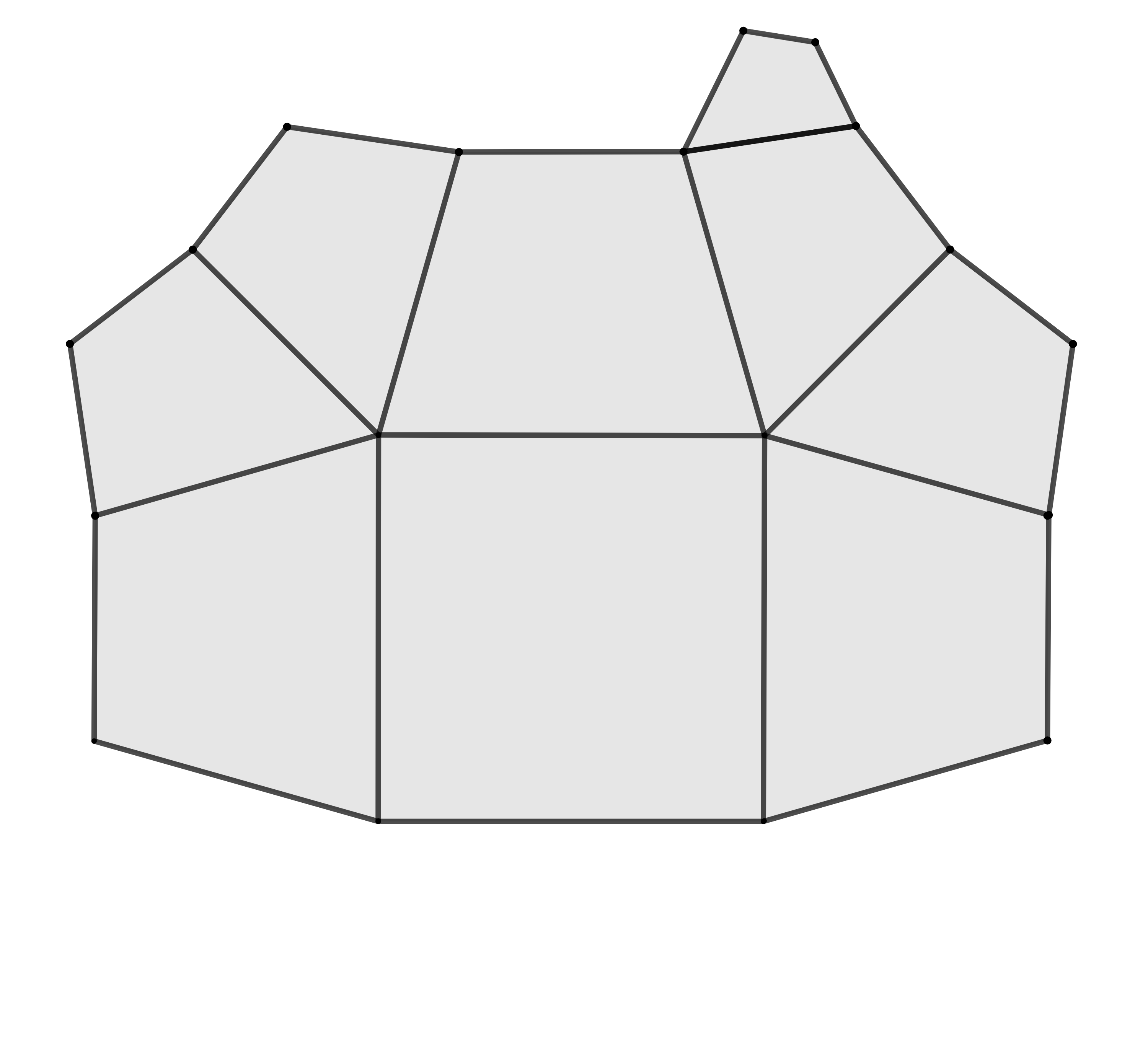}
   \includegraphics[scale=1.5]{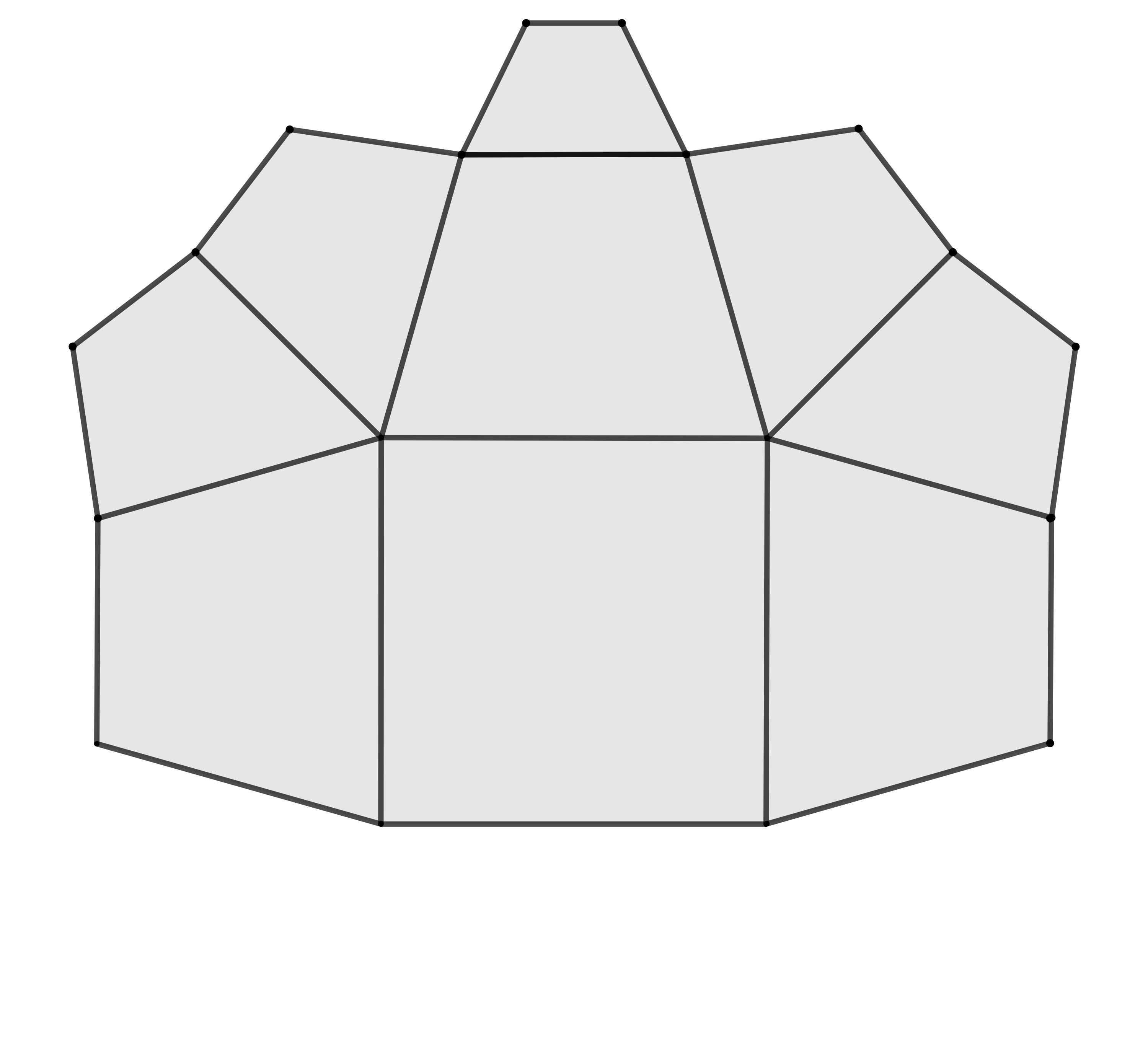}
   \includegraphics[scale=1.5]{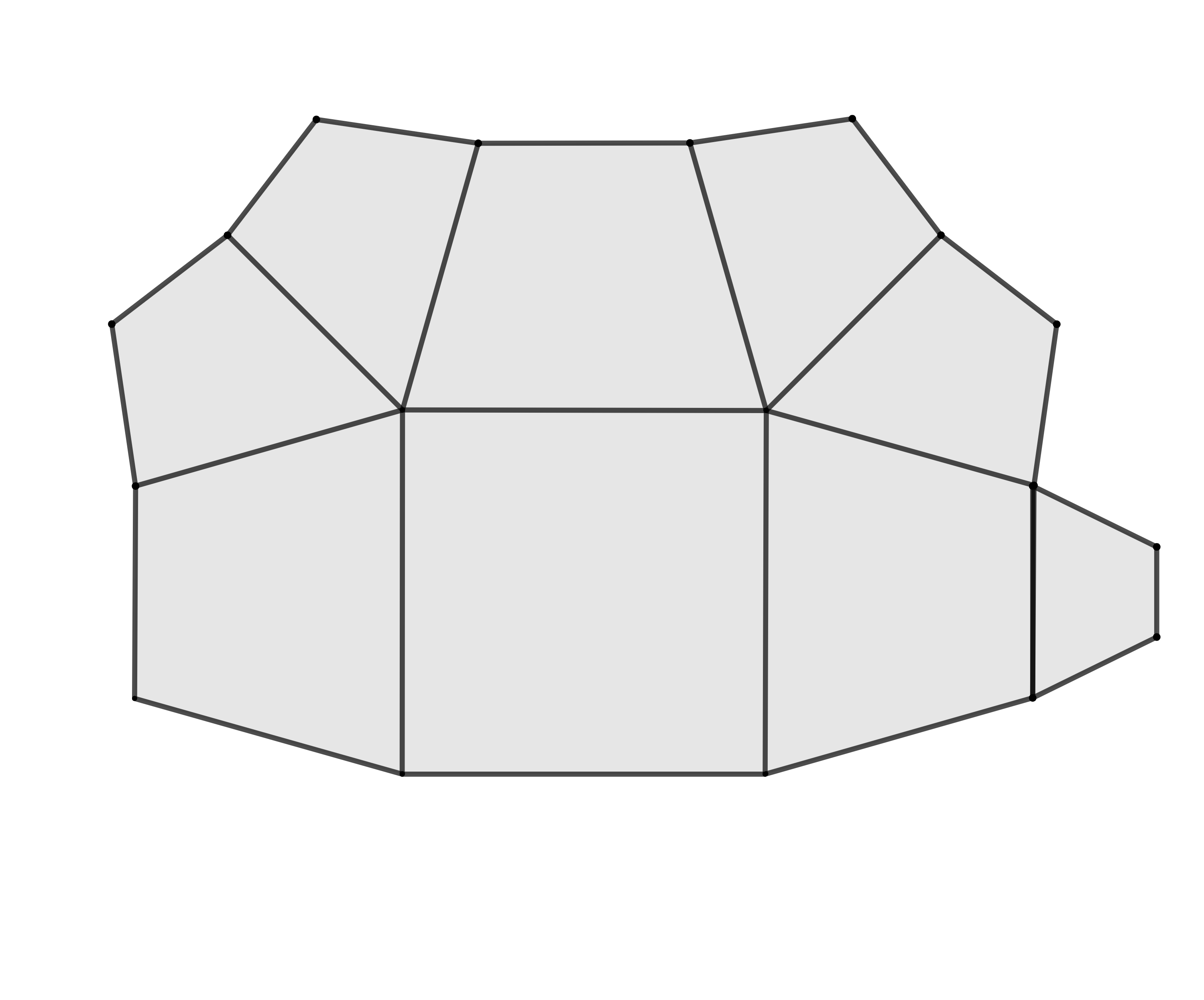}
   \includegraphics[scale=1.5]{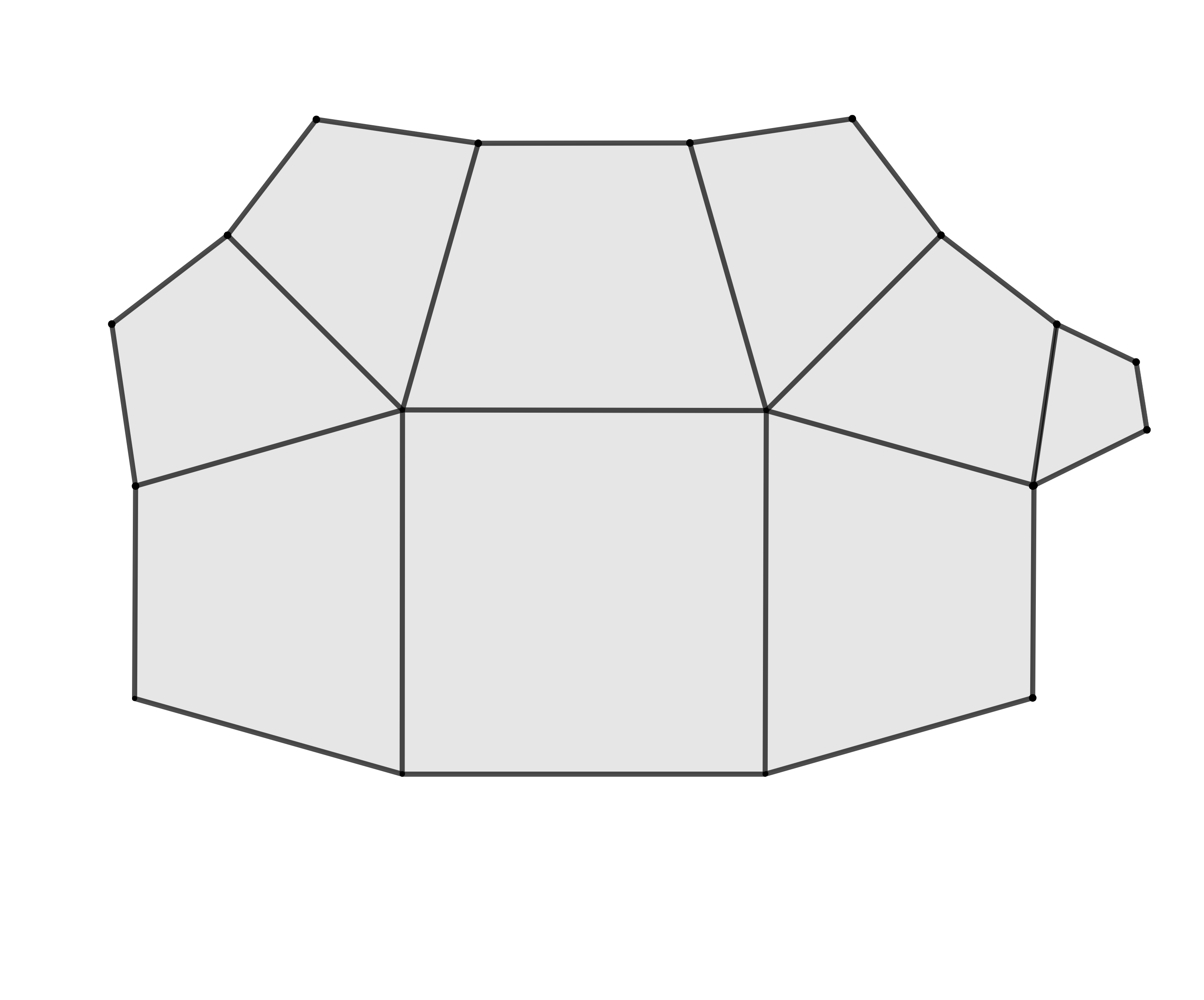}
   \includegraphics[scale=1.5]{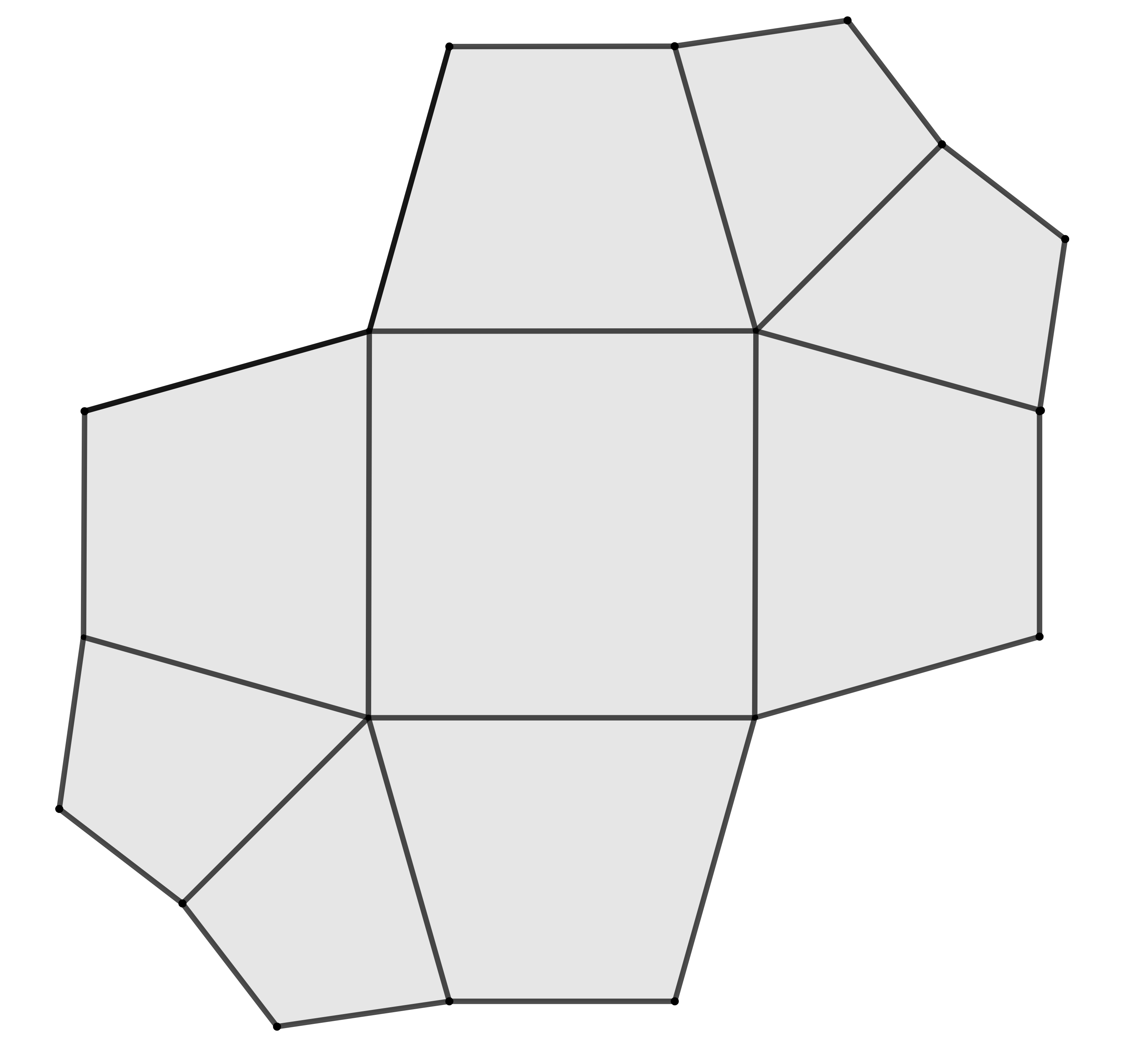}
    \caption{Free extremal $\{4,5\}$-animals with 9 tiles.} 
    \label{fig:extremal}
    \end{figure}

As shown in Figure \ref{fig:extremal}, it is not always the case that all $\{p,q\}$-spirals are \emph{unique} extremal animals. In the $\{4,5\}$-tessellation, for example, it is straightforward to construct an animal with $9$ tiles which is extremal, but not equivalent to $S_{4,5}(9)$, up to isometries. Note that in this setting $A_{p,q}(2)$ has 13 tiles, and tiles which are incident to only a vertex of $A_{p,q}(1)$ naturally appear in pairs in the corners. Then let $A$ be an animal obtained by removing two of these corner pairs. It is simple to compute that both $A$ and $S_{4,5}(9)$ will have $e_1=16$, and since they also have the same number of tiles, all other graph parameters will also be equal. Hence $A$ is also extremal. Depicted in Figure \ref{fig:extremal}, we see that there are exactly five distinct extremal $\{p,q\}$-animals with 9 tiles up to isometries, where the image on the far left is $S_{4,5}(9)$.

\section{Conclusions and open problems}
Recall that Theorem \ref{thm:extremalspiral} asserts that $\mathcal{P}_{p,q}(n)$, the minimum perimeter that an $n$-tile $\{p,q\}$-animal can attain, is equal to $e_1(n)$. Although here we give an algorithm to iteratively find the values of $\mathcal{P}_{p,q}(n)$, it remains as an open problem to find closed formulas for $\mathcal{P}_{p,q}(n)$ in the hyperbolic cases. The substitution rules seen in Section \ref{section:layers} hint at the theory of Sturmian words as a promising approach to finding such formulas. Nevertheless, we can derive some interesting properties of $\mathcal{P}_{p,q}$ as a direct consequence of Corollary \ref{cor:e_2bounds} and Theorem \ref{thm:extremalspiral}. For instance, we have that $\mathcal{P}_{3,q}(n)-\mathcal{P}_{3,q}(n) = \pm 1 $ for $q\ge 6$, $\mathcal{P}_{p,q}(n)$ is non-decreasing when $p=4$ or $\{p,q\}=\{6,3\}$, and $\mathcal{P}_{p,q}(n)$ is strictly increasing otherwise. These properties are summarized in the following Corollary. 
\begin{cor}\label{cor:increasing?}
For any $\{p,q\}$ with $(p-2)(q-2)\geq 4$,
\begin{align*}
&\mathcal{P}_{p,q}(n+1)-\mathcal{P}_{p,q}(n) = p-2 \text{ or $p-4$, when $q>3$.}  \\[4pt]
&\mathcal{P}_{p,3}(n+1)-\mathcal{P}_{p,3}(n)= p-4 \text{ or $p-6$, for $n>1$.}
\end{align*}   
\end{cor}
In the hyperbolic case, the minimum perimeter is of the order of the number of tiles. Bounds for $\mathcal{P}_{p,q}(n)$ for sufficiently large $n$ can be obtained from the following corollary to Theorem \ref{thm:e_1spiral} bounding $e_2(n)$.

\begin{cor} 
For $p\ge 4$ and $n>\nk(3)$, we have
    \begin{align*}
        n+\frac{n-1}{q-2} \le e_2(n) \le n+ \frac{n-1}{q-3}. 
    \end{align*}
For $p=3$ and $n>\nk(4)$, we have
    \begin{align*}
        n+\frac{n-1}{q-3} \le e_2(n) \le n+ \frac{n-1}{q-4}. 
    \end{align*}
\end{cor}
\begin{proof}
For $p \ge 4$, it can be checked that $\nk(k)+\dfrac{\nk(k)}{q-2} +1\le \ek_2(k) \le  \nk(k)+\dfrac{\nk(k)}{q-3}$, for $k\geq 3$. As usual, given $n$ choose $k$ such that $\nk(k)< n < \nk(k+1)$. In this case, each $d_{k,i}=2$ or $d_{k,i}=3$ so, Formula $(\ref{eqn:formula-m})$ gives the following bound
   \begin{align*}
        m(q-3) &\le \sum_{i=1}^m (q-d_{k,i}) \le n-\nk(k)-1 < \sum_{i=1}^{m+1} (q-d_{k,i})  \le (m+1)(q-2), \\
        \Longrightarrow \hspace{1cm}  & \frac{n-\nk(k)-1}{q-2}-1  \le m \le \frac{n-\nk(k)-1}{q-3} 
    \end{align*}
Then, the bounds for $\ek_2(k)$ and $m$, together with the formula $e_2(n)=n -\nk(k)+\ek_2(k)+m$ gives the result. The analysis for $p=3$ is analogous.
\end{proof}

We further conjecture a sharp limit for this ratio, based on the ratio for the $A_{p,q}(k)$ sequence.
\begin{conj}
$$ \lim_{n\to \infty} \frac{e_2(n)}{n} =   \lim_{k\to \infty} \frac{\ek_2(k)}{\nk(k)} =  \frac{q-1}{q-2} + \frac{1}{\alpha(q-2)}  $$
\end{conj}
Recall that $e_2(n)$ increments by the gluing parameter of each tile. So when $q=3$, all tiles are either 2-glued or 3-glued, and thus $$\dfrac{e_2(n)}{n}=\dfrac{2n+(\# \text{ of 3-glued tiles})}{n} = 2+\dfrac{\# \text{ of 3-glued tiles}}{n}.$$ And similarly when $q\ge 4$ all tiles are either 1-glued or 2-glued, and the ratio is $$\frac{e_2(n)}{n}=1+\dfrac{\# \text{ of 2-glued tiles}}{n}.$$ So this limit is capturing the ratio of these maximally glued tiles in each case. If this limit holds, then all such ratios for the graph parameters of $S_{p,q}(n)$ will also have limits given by the corresponding subsequences for $A_{p,q}(k)$.

One application of finding exact formulas for $\mathcal{P}_{p,q}(n)$ is to explore extremal topological animals. Euclidean animals with minimum perimeter have provided the right shape to produce animals with maximally many holes. We expect that the same will hold true for hyperbolic animals. Represent the maximum number of holes that a $\{p,q\}$-animal with $n$ tiles can have by $f_{p,q}(n)$. Then, using techniques introduced in \cite{MaR}, the following topological isoperimetric inequality holds
    \begin{align}
        f_{p,q}(n)\leq \frac{ n p - 2(n-1) -\mathcal{P}_{p,q}(n + f_{p,q}(n))}{p}.
    \end{align}
    This also provides a measure of certain geometric efficiencies, which are optimized when $f_{p,q}(n)$ attains this upper bound. In the Euclidean case, much of the analysis in solving the question of how many holes an $n$-tile animal can have stemmed from understanding these efficiencies, and analyzing the gap between $f_{p,q}(n)$ and this bound \cite{MaR, malen2021extremalI, malen2021extremalII}.

Finally, it is also natural to ask about algebraic characterizations of combinatorial parameters of higher dimensional extremal animals. For example, one can consider face-minimization properties of structures that are built by gluing together $d$-dimensional polytopes along codimension-1 faces in $\mathbb{R}^d$ or $\mathbb{H}^d$. In these settings, the analogs of the complete layered animals are straightforward to construct, but it is not immediately clear, for instance, how to generalize the 2-dimensional spirals that interpolate between them.

\subsection*{Acknowledgements}
This project received funding from the European Union's Horizon 2020 research and innovation program under the Marie Sk\l odowska-Curie grant agreement No.~754462.

\bibliographystyle{abbrvnat}
\bibliography{holeyominoes}

\end{document}